\documentclass[preprint,12pt,authoryear]{elsarticle}
\usepackage[margin=1in]{geometry}
\usepackage{amsmath,amssymb,mathtools,amsthm}
\usepackage{booktabs,longtable,tabularx,array}
\usepackage{pdflscape}

\usepackage{enumitem}
\usepackage{microtype}
\usepackage{natbib}
\usepackage{lmodern}
\usepackage{caption}
\usepackage[hidelinks]{hyperref}
\usepackage{algorithm}
\usepackage{algpseudocode}
\newtheorem{definition}{Definition}

\newtheorem{proposition}{Proposition}

\newcommand{\V}{\mathcal{V}}

\newcommand{\Hh}{\mathcal{H}}
\newcommand{\Tt}{\mathcal{T}}
\newcommand{\Q}{\mathcal{Q}}
\newcommand{\Pp}{\mathcal{P}}
\newcommand{\Nn}{\mathcal{N}}
\newcommand{\A}{\mathcal{A}}
\newcommand{\Am}{\mathcal{A}^{\mathrm{M}}}
\newcommand{\As}{\mathcal{A}^{\mathrm{S}}}
\newcommand{\Ar}{\mathcal{A}^{\mathrm{R}}}

\newcommand{\Az}{\mathcal{A}^{0}}
\newcommand{\Ecol}{E^{\mathrm{col}}}
\newcommand{\Ereq}{E^{\mathrm{req}}}

\newcommand{\eps}{\varepsilon}
\newcommand{\Rsep}{R^{\mathrm{sep}}}

\begin{document}

\begin{frontmatter}

\title{Noise-Aware Fleet Scheduling for On-Demand Urban Air Mobility with Alternative Flight Paths and Cruise Separation}

\author[label1]{Haoye Chen} 
\ead{haoye@kth.se}
\author[label2]{Jan Kronqvist}
\author[label1]{Wilco Burghout}
\author[label3]{Jianjun Wu}
\author[label1]{Zhenliang Ma\corref{cor1}}
\cortext[cor1]{Corresponding author.}
\ead{zhema@kth.se}
\affiliation[label1]{organization={Department of Civil and Architectural Engineering, KTH Royal Institute of Technology},
            city={Stockholm},
            postcode={10044}, 
            country={Sweden}}
\affiliation[label2]{organization={Department of Mathematics, KTH Royal Institute of Technology},
            city={Stockholm},
            postcode={10044}, 
            country={Sweden}}
\affiliation[label3]{organization={School of Economics and Management, Dalian University of Technology},
            city={Dalian},
            postcode={116024}, 
            country={China}}

\begin{abstract}
Multiple candidate flight paths between the same vertiports provide flexibility for managing noise exposure and cruise separation in on-demand urban air mobility (UAM). This paper introduces the fleet scheduling problem with alternative flight paths (FSP-AFP). We develop a time-expanded fleet-flow formulation for a finite aircraft fleet, in which each service or empty repositioning flight is assigned a planar path from a predefined set of alternatives connecting its origin and destination vertiports. These path choices are jointly optimized with cruise altitudes, departure times, and repositioning movements, while maintaining aircraft flow balance and enforcing geometric separation between flights cruising concurrently at the same altitude. The objective combines flight costs, service departure delays, and population-weighted noise exposure. A tailored branch-and-cut algorithm incorporates three complementary clique constructions based on a request-augmented conflict graph, exploiting physical flight conflicts, request exclusivity, and request–route relationships. Computational experiments show that, compared with a baseline implemented using Gurobi’s native lazy-constraint functionality, the proposed algorithm generally obtains stronger lower bounds and finds feasible solutions for all test instances, while generating fewer cuts on average and achieving a 43.16-fold reduction in the mean number of processed search nodes. An operational case study shows that increasing the number of candidate planar routes per vertiport pair from one to four reduces the weighted objective by 6.71\% and modeled noise exposure by 15.14\% under the cruise-separation constraints. Further increases in the number of candidate routes yield only marginal improvement in the objective value. These findings highlight the value of coordinating the use of alternative flight paths with departure scheduling and fleet circulation when evaluating the operational and environmental trade-offs of UAM.

\end{abstract}

\begin{keyword}
Urban air mobility \sep Alternative flight paths \sep Collision avoidance \sep Clique inequalities \sep Branch-and-cut
\end{keyword}

\end{frontmatter}
\section{Introduction}
Urban air mobility (UAM) is being developed as a potential complement to urban transportation, using aircraft to connect locations within and around metropolitan areas. Advances in electrification and aircraft design have renewed interest in passenger services operated between vertiports. Their development, however, depends on addressing operational and societal challenges, including airspace management, safety, noise, and community acceptance \citep{Cohen2021}. These challenges become closely connected when multiple aircraft operate over inhabited areas: flight paths determine where communities are exposed to noise, while the spatial and temporal organization of flights determines how aircraft can share the airspace without collisions. Effective operational planning is therefore essential to understanding how UAM services can accommodate transport demand while maintaining safety separation and limiting their environmental impact.

Route flexibility creates both opportunities and difficulties for this planning task. Alternative paths between the same vertiports can provide spatial separation between flights that would otherwise conflict, offering another choice instead of temporal separation, such as delaying departures. They can also support noise mitigation, although low-noise paths may require longer flight distances and travel times \citep{HILDEMANN2023102356}. These benefits cannot be assessed independently. A noise-motivated detour changes where and when an aircraft occupies the airspace, potentially removing a conflict along its original path while introducing another elsewhere. Similarly, a deviation intended to avoid one aircraft may bring the flight closer to other traffic. Several individually attractive paths may consequently be incompatible when operated simultaneously. Longer routes and conflict-induced departure delays also postpone aircraft availability, affecting subsequent passenger services and empty repositioning. The planning problem is thus to select a compatible combination of routes and schedules across the fleet, accounting for how detours redistribute both noise exposure and aircraft interactions.

These interactions raise a further question for UAM planning: how many alternative flight paths between each pair of vertiports are needed to obtain substantial operational and environmental benefits? Providing more paths creates additional routing and scheduling choices, but some may offer little benefit beyond the alternatives already available. Assessing their value therefore requires examining how successive additions affect noise exposure, flight distance, and service delay across the fleet. We investigate how aircraft separation requirements change the benefits of additional paths, and whether further route expansion can still reduce service delay after most of the noise-reduction benefits have been obtained.

To the best of our knowledge, this paper presents the first UAM fleet-scheduling formulation to integrate alternative planar flight-path selection with fleet flow conservation and continuous-time geometric cruise separation. We refer to this integrated problem as the fleet scheduling problem with alternative flight paths (FSP-AFP). For each vertiport pair, we consider multiple spatially distinct paths that connect the two vertiports directly, without requiring flights to follow a shared network of predefined intermediate junctions and links. These flight paths are jointly optimized with cruise altitudes and departure times for service and empty repositioning flights, minimizing a weighted combination of flight costs, request delays, and population-weighted noise exposure. Alternative paths provide opportunities for noise mitigation and spatial conflict avoidance, while fleet-flow constraints connect successive aircraft activities and account for their effects on subsequent services and repositioning movements. A longer path may therefore support a better overall schedule, whereas a low-noise path may be rejected because of its conflicts with other flights or its effects on aircraft availability. The large number of flight alternatives, however, gives rise to numerous pairwise cruise-separation constraints, whose explicit formulation can impose prohibitive memory requirements. We address this challenge by generating separation constraints dynamically as needed, avoiding the enumeration of the full pairwise constraint set. To improve computational performance, we develop a tailored branch-and-cut algorithm with selective clique separation. The resulting clique inequalities exploit both physical flight conflicts and the mutual exclusivity of alternatives serving the same request to strengthen the linear programming relaxation.

This study connects research on the spatial and environmental dimensions of UAM with research on fleet operations. Noise-aware traffic allocation and flight-path planning have advanced the treatment of community exposure and spatial route selection \citep{GAO2024104740, CHEN2026112176}. Complementary studies have integrated fleet and infrastructure planning with aircraft operating decisions, and coordinated passenger service and rebalancing under safety and energy requirements \citep{SHON2025102720, Pooladsanj2026}. However, the interaction between multiple direct paths for each vertiport pair and the service–repositioning sequences of a finite fleet remains insufficiently examined both in modeling and analysis when noise exposure and geometric cruise separation are considered together. Methodologically, alternative paths expand the set of route–altitude–departure combinations, whose selection is governed by physical conflicts and the coverage that each request be served exactly once. In the linear programming relaxation, pairwise conflict constraints, even together with request coverage, may admit fractional solutions that violate clique inequalities. This motivates relaxation strengthening that exploits both sources of incompatibility. The computational challenge is to identify such strengthening inequalities while controlling the effort spent on geometric conflict detection, cut generation, and repeated linear programming solves.

The contributions of this study are threefold:
    \begin{itemize}
        \item An integrated model and incompatibility-graph representation for on-demand UAM routing and operations. We jointly optimize service and rebalancing flights over multiple direct paths between vertiports, accounting for noise exposure. A request-augmented incompatibility graph combines cruise conflicts with the mutual exclusivity of flight alternatives serving the same request.
        \item A branch-and-cut method exploiting flight and request structure. We develop a selective clique-separation scheme that exploits these incompatibilities to identify and strengthen violated inequalities. Computational experiments benchmark the resulting algorithm against a baseline using Gurobi's native lazy-constraint functionality.
        \item Operational insights into the value of route flexibility. A controlled case study in Stockholm shows that nearly all observed noise-exposure improvements are achieved with a small set of candidate paths per vertiport pair. Further expansion yields only marginal improvement in the weighted objective value.

    \end{itemize}

The rest of the paper is organized as follows: Section \ref{Sec: Literature review} introduces the research gap on routing, modeling, and collisions of UAM systems. Section \ref{sec:methodology} formulates the model with cruise separation. Section \ref{sec:cliques} introduces the clique generation methods and branch-and-cut algorithm. Section \ref{sec:experiments} conducts numerical and operational experiments. Section \ref{sec: conclusion} summarizes the findings, applications, and future directions.

\section{Literature review}
\label{Sec: Literature review}
\subsection{Airspace Modeling and Noise-Aware Routing}
UAM network and operational models differ in how they represent flight connections between vertiports. \cite{GAO2024104740} constructed a multilayer aerial node–link network, with vertiports and airborne waypoints as nodes and flight corridors as links, to allocate traffic while considering community noise and equity. Their model permits multiple routes for each Origin--Destination (OD) pair, but these alternatives are different sequences of connected links through a shared link network, similar to a road network. This differs from explicitly providing several direct flight paths between the same two vertiports, each of which may follow a different geometric detour without having to pass through other vertiports or intersections. \cite{Pooladsanj2026} likewise represented UAM airspace as a directed graph whose nodes correspond to vertiports or route intersections and whose links represent potentially shared sections of flight routes. Their VertiSync policy jointly schedules passenger-serving and rebalancing flights over these connected route sections, while the presented analysis assumes at most one route per OD pair. Other fleet-operation models abstract inter-vertiport movements through OD connections with given travel times: \cite{SHON2025102720} integrated fleet and infrastructure planning with aggregate relocation flows, while \cite{Guo2026} jointly optimized flight assignment, service and deadhead movements, charging, and takeoff–landing operations; both formulations use given OD travel times without selecting among geometrically distinct flight paths. \cite{zhu2019pre} explicitly recognized the advantages of direct vertiport connections and allowed flights to choose among multiple route alternatives while adjusting delays and speeds. Their framework nevertheless processes flight plans sequentially and omits connections between successive flights at vertiports. Together, these studies motivate a gap that retains multiple direct flight paths between each vertiport pair and jointly evaluates their selection with request assignment, departure scheduling, aircraft reuse, and empty rebalancing.

Spatially explicit noise- and risk-aware routing studies provide a complementary perspective by using environmental maps to identify beneficial detours. \cite{KYI2026100231} employed grid-based modified A* routing with population-weighted noise costs. \cite{CHEN2026112176} generated multiple candidate paths over gridded ground-risk and noise maps and subsequently refined high-noise segments. It further incorporated multi-aircraft coordination and fleet-wide path-combination selection, although aircraft missions, departure times, and planning priorities were specified beforehand. \cite{Korringa2025} similarly combined spatial population, communication-coverage, and air-traffic data with A* search to design low-risk AAM corridors. From a delivery-routing perspective, \cite{Tan2024} integrated a flight-noise simulator with customer allocation and depot-based tour optimization. These studies demonstrate the environmental value of spatial route choices and address different aspects of multi-aircraft routing. However, their formulations do not jointly optimize request assignment, departure scheduling, aircraft reuse across successive OD services, and empty repositioning between vertiports. Consequently, the environmental benefits of a path are not evaluated together with the complete sequence of fleet movements needed to serve subsequent requests.

A remaining gap in the reviewed literature is therefore the integration of multiple direct flight-path alternatives between each vertiport pair with environmentally informed fleet scheduling and rebalancing. These alternatives connect the origin and destination without requiring passage through other vertiports or virtual intersections in the air, while allowing different geometric detours.

\subsection{Collisions}
A closely related research area is aircraft conflict resolution within conventional air traffic management, where trajectory adjustments are optimized to maintain required separation between aircraft. \cite{LEHOUILLIER2017696AirCollsionClique} formulated the problem as a minimum-weight maximum-cardinality clique problem, with vertices representing discrete maneuvers and edges linking mutually compatible maneuvers of different aircraft. They developed sequential MILP and flight-level-based large-neighborhood decomposition methods to improve computational scalability. \cite{DIAS2022520} developed nonconvex mixed-integer formulations with continuous speed and heading controls using disjunctive linear separation conditions. They further incorporated flight-level reassignment through a two-step formulation that prioritizes minimizing flight-level changes before resolving the remaining conflicts through horizontal velocity adjustments. \cite{CAFIERI202367Continous} considered simultaneous speed and heading adjustments at a common fixed flight level. They proposed an MINLP formulation for minimizing maneuver deviations and a continuous nonlinear penalty formulation whose zero-penalty solutions satisfy the separation requirements. A broader review of mathematical programming approaches to aircraft conflict resolution is provided in \cite{TS_ATC_review}. However, existing aircraft conflict resolution research assumes a fixed set of flights with predetermined nominal trajectories and focuses on the moving dynamics and tactical separation maneuvers of multiple airplanes. They do not jointly consider passenger assignment, demand fulfillment, fleet scheduling, or empty-vehicle rebalancing, which are common contexts in urban mobility.

Building on conventional air traffic management, studies of UAM safety have addressed strategic flight planning, tactical collision avoidance, and collision-risk assessment. \cite{TANG2021103324} combined obstacle-aware route generation with flight-level assignment to obtain conflict-free 4D flight plans while considering operating costs and fairness among operators. Their model enforces separation at route intersections and overlapping segments, but takes flight requests and departure times as inputs without representing aircraft reuse or empty repositioning across successive services. \cite{Chour2023} jointly selected delivery routes and departure times through a hybrid Monte Carlo tree search framework, balancing ground delay, flight distance, and ground risk under separation requirements. However, aircraft are assumed to be continuously available at depots and leave the modeled system after delivery, so return flights and subsequent assignments do not constrain scheduling decisions. \cite{haba2025routing} formulated simultaneous route selection for eligible flight requests as a maximum weighted independent set problem, accounting for conflicts among candidate routes and with ongoing flights. Their framework selects a compatible batch for immediate departure at each decision epoch, while previously approved trajectories remain fixed. Although each batch can be solved using an exact integer-programming solver, departure times are not jointly optimized across the planning horizon, and aircraft reuse and empty repositioning are not represented. Beyond predeparture planning, \cite{WENG2026105477} combined centralized route guidance with distributed velocity-obstacle collision avoidance for given flight missions, without modeling successive service and repositioning movements. Complementing these planning and control approaches, \cite{Su2025} linked mid-air encounters and mitigation failures to subsequent impacts and fatalities under operational uncertainties. Taken together, these studies advance flight-level planning, control, and risk assessment, but do not jointly schedule passenger service and empty repositioning while maintaining aircraft availability and continuity across successive operations.

More directly linking separation requirements with fleet operations, \cite{Pooladsanj2026} developed VertiSync to jointly schedule passenger-serving and rebalancing movements while maintaining aircraft continuity. Their formulation represents UAM airspace as a directed corridor network and enforces airborne separation through capacity-one sectors, assuming that distinct links are safely separated except near intersections. The presented analysis further assumes at most one route per OD pair. A different modeling setting arises when each vertiport pair has multiple candidate flight paths whose geometric separation is not guaranteed by the network layout. In this setting, concurrent same-altitude flights on nearby paths may violate the prescribed separation distance even when the paths neither intersect nor share a link. The resulting modeling challenge is to coordinate these geometric separation requirements with route choices and the successive service and repositioning movements of a finite fleet.

Conflict graphs provide a general framework for strengthening integer programming formulations through clique inequalities. \cite{ATAMTURK200040} developed conflict-graph techniques for exploiting binary incompatibilities and generating cutting planes. In railway station routing, \cite{ZWANEVELD200114} represented incompatible routing choices in a node-packing formulation and incorporated violated clique inequalities, extended to maximal cliques, into branch-and-cut. For project scheduling, \cite{ARAUJO2020104782} constructed dynamic conflict graphs and combined clique cuts with lifted cover and precedence inequalities. \cite{BRITO2021105176} controlled separation effort through bounded searches over fractional-support graphs, followed by greedy clique extension. More recently, \cite{MOURA2025433} strengthened a representatives formulation for parallel-machine scheduling with stable-set inequalities, with their experiments favoring exact clique separation at the root node. These studies demonstrate that practical clique separation involves balancing relaxation strengthening against the computational cost of finding and adding inequalities. A remaining challenge is to efficiently identify useful clique inequalities by exploiting the interaction between physical flight conflicts and request exclusivity in UAM systems that jointly schedule passenger-serving and empty repositioning flights over multiple direct paths between vertiport pairs.

\section{Model}
\label{sec:methodology}
We consider a homogeneous fleet of drones serving requests between vertiports. Each request specifies an origin, a destination, and a release time. A service flight transports one request directly between its specified vertiports. Empty flights may reposition drones to serve subsequent requests or restore the terminal fleet distribution. For each ordered OD pair, a set of horizontal routes is prescribed in advance. Each flight selects one of these routes and an available cruise altitude.

Every flight is a complete ground-to-ground leg: vertical takeoff, horizontal cruise, and vertical landing. A drone becomes available at the destination when landing is complete and may select a different cruise altitude for its next leg. Drones wait on the ground at vertiports. Ground waiting has no direct cost; a request incurs a delay cost when its flight departs after its release time. At the end of the planning horizon, each vertiport must recover its initial number of drones. This requirement concerns only inventories, so an individual drone need not return to its own starting vertiport.

The separation requirement of safety distance applies only to flights simultaneously cruising at the same altitude. Takeoff and landing occupy a drone for their specified durations but are excluded from this separation test. The cruise-separation requirement applies to pairs of service flights, pairs of rebalancing flights, and pairs consisting of one flight of each type. Ground parking and takeoff-pad capacities are not modeled, and no additional turnaround or charging time is assumed.

Service and rebalancing flights are planned jointly. Rebalancing flights reposition drones between vertiports without serving requests, changing fleet availability and potentially altering the departure times, routes, and cruise altitudes of service flights. The model may accept the additional cost of a rebalancing flight to position a drone for an earlier service departure or a tighter connection between successive services to obtain a better and feasible flight plan. All requests and initial fleet inventories are given in advance, and the complete schedule is constructed without a reference schedule or a schedule-deviation penalty.

\begin{figure}[ht!]
\centering
\includegraphics[width=\linewidth,height=0.7\textheight,keepaspectratio]{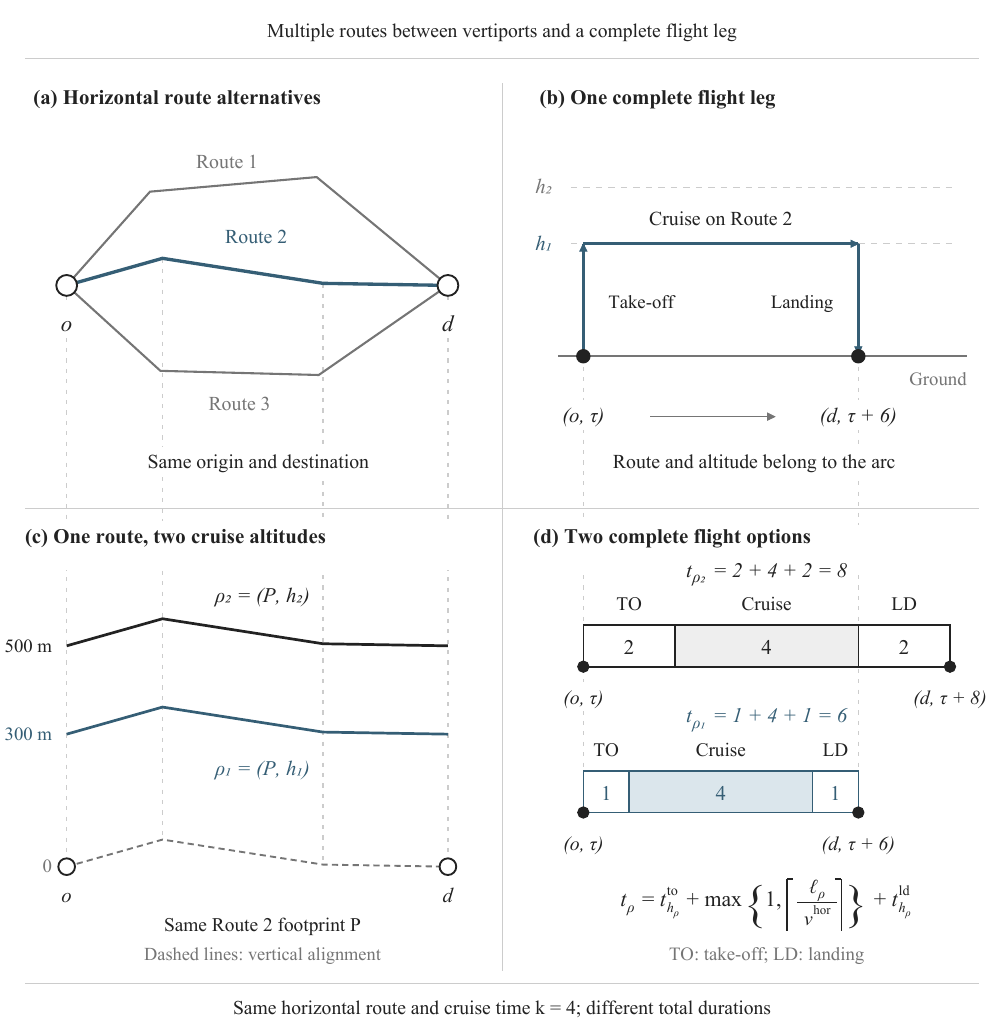}
\caption{Multiple routes between vertiports and a complete flight leg.}
\label{fig:multiple-routes}
\end{figure}

\subsection{Sets, parameters, and the expanded network}
\label{subsec:physical-setting}

Let \(\V\) be the set of vertiports, \(\Hh\) the set of available altitude layers, and $\Tt=\{0,1,\ldots,T\}$ the discrete planning horizon. A request $q\in\Q$ is the tuple $q=(o_q,d_q,r_q)$, where \(o_q,d_q\in\V\), \(o_q\neq d_q\), and \(r_q\in\Tt\) is the release time of the request.

For each directional OD pair \((o,d)\),  let $\Pp_{od}$ contain its route--altitude options. An option $\rho\in\Pp_{od}$ specifies a cruise altitude $h_\rho\in\Hh$, a horizontal polyline $P_\rho=(p_{\rho,0},\ldots,p_{\rho,m_\rho})$ of length $\ell_\rho>0$, a route--altitude option cost $C_\rho$, and integer takeoff and landing durations $t^{\mathrm{to}}_{h_\rho}$ and $t^{\mathrm{ld}}_{h_\rho}$. Thus, a horizontal route available at two altitudes gives two distinct options. The supplied cost $C_\rho$ includes the distance, noise exposure, and altitude-related components assigned to that option. With horizontal speed \(v^{\mathrm{hor}}\), its discrete duration is
 $ t_\rho
  =
  t^{\mathrm{to}}_{h_\rho}
  +\max\left\{1,\left\lceil
       \frac{\ell_\rho}{v^{\mathrm{hor}}}
     \right\rceil\right\}
  +t^{\mathrm{ld}}_{h_\rho}.$
 Figure~\ref{fig:multiple-routes} illustrates the route choice and its representation as a complete flight leg.

{
The time-expanded network has node set
\begin{equation}
 \Nn=(\V\times\Tt).
 \label{eq:state-space}
\end{equation}
A state $(v,t)$ represents drones available on the ground at vertiport $v$
at time $t$. Cruise altitude is an attribute of a moving arc, not of a node.
Let $D_v\in\mathbb Z_{\geq0}$ be the initial inventory at $v$ and $D=\sum_{v\in\V}D_v$ the fleet size. Assume the same initial and final inventory at every vertiport. The indexed notation below also states the condition. The net supply is
\begin{equation}
 b_n=\begin{cases}
 D_v, & n=(v,0),\quad v\in\V,\\
 -D_v, & n=(v, T),\quad v\in\V,\\
 0, & \text{otherwise}.
 \end{cases}
 \label{eq:supply}
\end{equation}}

\paragraph{Arc classes}
\label{subsec:arc-pool}

The arc set is partitioned as
\begin{equation}
 \A=\As\cup\Ar\cup\Az,\qquad
 \Am=\As\cup\Ar.
 \label{eq:arc-partition}
\end{equation}
The moving arcs $\Am$ represent service and rebalancing flights. Structural arcs $\Az$ represent ground waiting. Each arc $a \in \A$ has tail $u_a$, head $v_a$, departure $\tau_a$, arrival $\bar\tau_a$, and cost $c_a$. A moving arc also records its route option $\rho_a$, altitude $h_a=h_{\rho_a}$, and trajectory $P_\rho$. Parallel arcs remain distinct when their requests, routes, altitudes, or departure choices differ, even if they share ground endpoints. Note that takeoff and landing belong to moving arcs.

\paragraph{Service arcs $\As$}
For each request $q\in\Q$, route--altitude option
$\rho\in\Pp_{o_qd_q}$, and departure time $\tau\in\Tt$
satisfying $\tau\ge r_q$ and $\tau+t_\rho\le T$,
the service arc $a=(q,\rho,\tau)$ has
\begin{align}
  u_a &= (o_q,\tau), \\
  v_a &= (d_q,\tau+t_\rho),
  \label{eq:service-incidence}\\
  c_a &= C_\rho+c^{\mathrm{delay}}(\tau-r_q),
  \label{eq:service-cost}
\end{align}
where $c^{\mathrm{delay}}\geq0$ is the cost per unit of departure delay. Departure marks the start of takeoff, and arrival marks the completion of landing. Let \(\As_q\subseteq\As\) denote all service arcs for request \(q\), covering every route option \(\rho\in\Pp_{o_qd_q}\) and every departure time \(\tau\in\Tt\) satisfying \(r_q\le\tau\le T-t_\rho\).

\paragraph{Rebalancing arcs $\Ar$}
{
For $o\ne d$, $\rho\in\Pp_{od}$, and $\tau\in\Tt$ with $\tau+t_\rho\leq T$, a rebalancing arc connects $(o,\tau)$ to $(d,\tau+t_\rho)$. It occupies one drone for the complete flight but serves
no request.} Its cost is
\begin{equation}
 c_a=\gamma^{\mathrm{rb}}C_\rho,
 \label{eq:rebal-cost}
\end{equation}
where $\gamma^{\mathrm{rb}}\geq0$ is the rebalancing cost multiplier.

\paragraph{Ground-waiting arcs $\Az$}
{
For $v\in\V$ and $t<T$, an arc $(v,t)\longrightarrow(v,t+1)$ carries ground inventory at zero cost. Waiting after a drone's final flight is represented in the same way until $T$ and generates no request-delay charge. The ground-waiting arcs serve as structural arcs connecting moving arcs.}

\begin{table}[ht!]
\centering
\caption{Principal sets and parameters.}
\label{tab:set-notation}
\label{tab:parameter-notation}
\small
{
\begin{tabularx}{\textwidth}{@{}p{0.26\textwidth}X@{}}
\toprule
Symbol & Definition \\
\midrule
$\V,\Hh,\Tt,\Q$ & Vertiports, cruise altitudes, discrete times, and requests. \\
$\Pp_{od}$ & Route--altitude options for ordered pair $(o,d)$. \\
$\Nn$ & Ground states $(v,t)$ that $v \in \V, \  t \in \Tt$. \\
$\As,\Ar,\Az$ & Service, rebalancing, and ground-waiting arcs. \\
$\Am,\Az$ & Moving arcs $\As\cup\Ar$ and structural arcs $\Az$. \\
$\As_q$ & Service alternatives for request $q \in \Q$. \\
$T,D_v,D$ & Last planning time, initial and final inventory at $v$, and fleet size. \\
$(o_q,d_q,r_q)$ & Origin, destination, and release time of request $q \in \Q$. \\
$(h_\rho,P_\rho,\ell_\rho,C_\rho)$ & Cruise altitude, polyline, length, and cost of option $\rho \in \Pp_{od}$. \\
$t_h^{\mathrm{to}},t_h^{\mathrm{ld}}$ & Takeoff and landing durations for cruise altitude $h \in \Hh$. \\
$v^{\mathrm{hor}},t_\rho$ & Reference cruise speed and full leg duration. \\
$c^{\mathrm{delay}},\gamma^{\mathrm{rb}}$ & Departure-delay cost and rebalancing cost multiplier. \\
$\Rsep$ & Required horizontal separation during same-altitude cruise. \\
\bottomrule
\end{tabularx}}
\end{table}

Table \ref{tab:set-notation} summarizes the sets and parameters used in the model.

\subsection{Formulation}
\subsubsection{Decision variables}

Table \ref{tab:arc-classes} describes the decision variables. For each moving arc $a\in\Am$, let moving variable $x_a\in\{0,1\}$ indicate its selection. For each structural arc $a\in\Az$, let $f_a$ be its drone flow. We use continuous structural variables with $0\leq f_a\leq D$ for $a \in \Az$. This relaxation does not change the optimal value, as Proposition \ref{prop:structural-tu} shows. Let $x=(x_a)_{a\in\Am}$ and $f=(f_a)_{a\in\Az}$ denote the vectors of moving and waiting-flow variables, respectively. 

\begin{table}[ht!]
\centering
\caption{Arc classes and variable domains.}
\label{tab:arc-classes}
\label{tab:variables}
\small
{
\begin{tabularx}{\textwidth}{@{}p{0.12\textwidth}p{0.24\textwidth}p{0.24\textwidth}X@{}}
\toprule
Arc set & Indexing & Variable domain & Cost \\
\midrule
$\As$ & Request, option, departure & $x_a\in\{0,1\}$ & $C_{\rho_a}+c^{\mathrm{delay}}(\tau_a-r_q)$ \\
$\Ar$ & Option, departure & $x_a\in\{0,1\}$ & $\gamma^{\mathrm{rb}}C_{\rho_a}$ \\
$\Az$ & Vertiport, time & $0\leq f_a\leq D$ & $0$ \\
\bottomrule
\end{tabularx}}
\end{table}

{
\begin{proposition}[Structural-flow integrality]
\label{prop:structural-tu}
Fix the moving variables to an integral vector $x$. If the remaining problem in the waiting-flow variables $f$ after fixing $x$ is feasible, it has an integral optimal solution when the structural variables are continuous.
\end{proposition}

\begin{proof}
The residual conservation equations are $B_0f=b-B_Mx$, where $B_0$ and $B_M$ are the node--arc incidence matrices of the structural and moving arcs, respectively. Since there is no additional constraint on structural arcs except the basic flow conservation, the matrix $B_0$ is totally unimodular. Its right-hand side is integral because the initial supplies, sink demands, and fixed moving variables are integral. All structural bounds $0\leq f_a\leq D$ are integral as well. The resulting bounded network-flow polyhedron has only integral extreme points. A linear objective therefore attains an integral optimum whenever the residual problem is feasible.
\end{proof}}

This proposition does not establish total unimodularity of the full model: request coverage and collision constraints couple the moving arcs outside the network incidence structure. It allows waiting-flow variables to remain continuous and not require branching. Service and rebalancing variables must still be binary, relaxed in each node LP, and have their integrality enforced through branching. Once the moving arcs are fixed to zero or one, the remaining waiting-flow problem has an integral optimal solution whenever it is feasible.

\subsubsection{Collision condition}
\label{subsec:collision}
We consider collision only at the cruise stage represented by moving arcs. Specifically, the separation rule is applied only while a drone traverses the horizontal route polyline $P_\rho=(p_{\rho,0},\ldots,p_{\rho,m_\rho})$ for route-altitude option $\rho$. For a moving arc \(a\), define its cruise interval
\begin{equation}
  I_a=
  [\tau_a+t^{\mathrm{to}}_{h_a},
   \bar\tau_a-t^{\mathrm{ld}}_{h_a}].
  \label{eq:cruise-interval}
\end{equation}
Let \(\gamma_a(t)\in\mathbb{R}^2\) be the piecewise-linear position along the polyline, parameterized at constant horizontal speed over \(I_a\).

\begin{definition}[Exact collision]
\label{def:collision}
Two distinct moving arcs \(a,b\in\Am\) conflict when
\begin{equation}
  h_a=h_b,\qquad
  I_a\cap I_b\neq\varnothing,\qquad
  \min_{t\in I_a\cap I_b}
  \|\gamma_a(t)-\gamma_b(t)\|_2
  <\Rsep.
  \label{eq:collision-predicate}
\end{equation}
Let $\Ecol$ be the set of unordered pairs satisfying this condition.
\end{definition}

We first use trajectory bounding boxes, together with altitude and time-overlap checks, to identify potentially colliding arc pairs, substantially reducing the number of pairs requiring exact evaluation. The minimum in \eqref{eq:collision-predicate} is computed exactly for every overlapping pair of linear motion segments by minimizing a convex quadratic in their common time interval.

\subsubsection{Formulation}
\label{subsec:exact-formulation}

{
For $n\in\Nn$, let $\delta^+(n)=\{a\in\A:u_a=n\}$ and
$\delta^-(n)=\{a\in\A:v_a=n\}$. The FSP-AFP formulation is:}

\begin{subequations}
\label{model:exact}
\begin{align}
 \min\quad &\sum_{a\in\Am}c_ax_a
 &&\label{eq:objective}\\
 \text{s.t.}\quad
 &\sum_{a\in\delta^+(n)\cap\Am}x_a
  +\sum_{a\in\delta^+(n)\cap\Az}f_a
  -\sum_{a\in\delta^-(n)\cap\Am}x_a
  -\sum_{a\in\delta^-(n)\cap\Az}f_a=b_n
 &&\forall n\in\Nn,\label{eq:flow}\\
 &\sum_{a\in\As_q}x_a=1
 &&\forall q\in\Q,\label{eq:coverage}\\
 &x_a+x_b\leq1
 &&\forall\{a,b\}\in\Ecol,\label{eq:pairwise-collision}\\
 &x_a\in\{0,1\}
 &&\forall a\in\Am,\label{eq:service-rebalance-domain}\\
 &0\leq f_a\leq D
 &&\forall a\in\Az.\label{eq:structural-domain}
\end{align}
\end{subequations}

Objective \eqref{eq:objective} summarizes the cost associated with moving arcs. Each arc contains distance, noise exposure, altitude, and request delay costs. Conservation equations~\eqref{eq:flow} route the fleet from its initial supplies to its final sinks. Request coverage constraints~\eqref{eq:coverage} ensure every request is served once, and collision constraints~\eqref{eq:pairwise-collision} exclude physically conflicting flight pairs. The pairwise formulation defines the full feasible set. 

Multiple requests, route--altitude options, and departure times give rise to a large set of moving arcs. The number of pairwise collision constraints can grow quadratically with the number of moving arcs. Including all such constraints explicitly can be computationally expensive and may exceed available memory. In the algorithm, collision
constraints are generated as needed and strengthened by the clique inequalities developed next.

\section{A clique-based branch-and-cut algorithm}
\label{sec:cliques}
There are two main challenges in solving model \eqref{model:exact}: the large number of collision constraints and a weak LP relaxation that can lead to a large search tree. We address the first by generating collision constraints only when needed. To strengthen the relaxation, we construct clique inequalities using physical conflicts and incompatibilities among alternatives for the same request. These inequalities are combined with branching rules that exploit request structure and congestion information in a tailored branch-and-cut algorithm.

This section is organized as follows: Section \ref{subsubsec:augmented-incompatibility-graph} introduces the incompatibility graphs. Section \ref{subsec:clique-methods} develops three constructions of clique inequalities based on these graphs. Section \ref{subsec:separation-acceleration} describes the separation procedure. Section \ref{subsec:branching} presents the branch-and-cut algorithm, including its node selection and branching rules.

\subsection{Exact and request-augmented incompatibility graphs}
\label{subsubsec:augmented-incompatibility-graph}
{The separation procedure uses a graph whose vertices are individual moving arcs. Service alternatives for different requests remain distinct vertices even when they have identical trajectories. Let $G^{\mathrm{col}}=(\Am,\Ecol)$ denote the physical collision graph, with edges defined by the continuous-time condition in Definition~\ref{def:collision}.}

\begin{figure}[ht!]
  \centering
  \includegraphics[width=\linewidth]{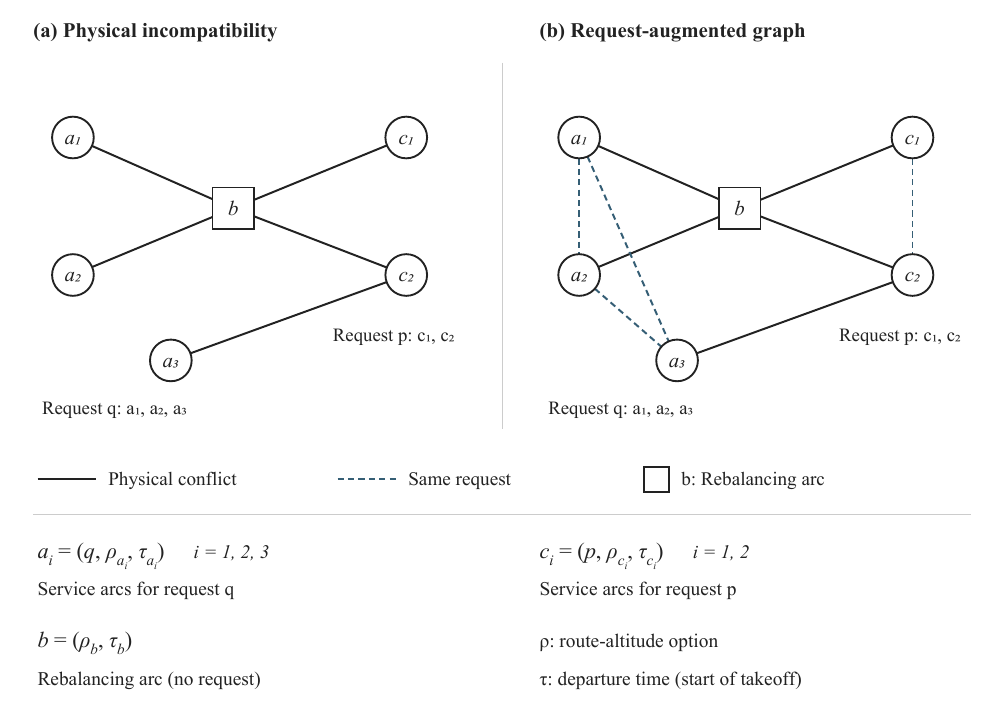}
  \caption{A physical collision graph and its request-augmented
  counterpart.}
  \label{fig:request-augmented}
\end{figure}

Request coverage introduces logical incompatibility between distinct service alternatives for the same request. We define the corresponding edge set as
\begin{equation}
  \Ereq
  = \bigcup_{q\in\Q}
    \bigl\{\{a,b\}:a,b\in\As_q,\ a\ne b\bigr\}.
  \label{eq:request-edges}
\end{equation}
Combining these logical incompatibilities with physical conflicts gives the request-augmented incompatibility graph
\begin{equation}
  G^+=(\Am,\Ecol\cup\Ereq).
  \label{eq:augmented-graph}
\end{equation}
The two edge sets need not be disjoint: two service alternatives for the same request may also conflict physically. Figure~\ref{fig:request-augmented} illustrates how the same-request edges augment the physical collision graph.

Since the endpoints of every edge in $G^+$ are mutually exclusive, the moving arcs selected by any feasible integer solution form a stable set of $G^+$. Consequently, at most one arc can be selected from any clique $K$ of $G^+$, yielding the valid inequality $\sum_{a\in K}x_a\leq 1$. This observation underlies the clique inequalities developed below.

\paragraph{Illustrative example}
{Consider one request $q$ with alternatives $\As_q=\{a_1,a_2\}$ and a rebalancing arc $b$. Suppose $\{a_1,b\},\{a_2,b\}\in\Ecol$ but $\{a_1,a_2\}\notin\Ecol$. The physical graph is a path. The request edge $\{a_1,a_2\}$ completes a triangle in $G^+$, yielding}
\begin{equation}
 x_{a_1}+x_{a_2}+x_b\le1.
 \label{eq:example-augmented-clique}
\end{equation}
{The point $\bar x_{a_1}=\bar x_{a_2}=\bar x_b=1/2$ satisfies $\bar x_{a_1}+\bar x_{a_2}=1$ and both physical pairwise inequalities, but its activity in~\eqref{eq:example-augmented-clique} is $3/2$. Indeed, coverage and the clique inequality together imply $x_b=0$.}
{This example isolates the strengthening of the coverage-and-conflict relaxation as shown by Figure~\ref{fig:triangle-example}.}

\begin{figure}[h!]
 \centering
 \includegraphics[width=\linewidth]{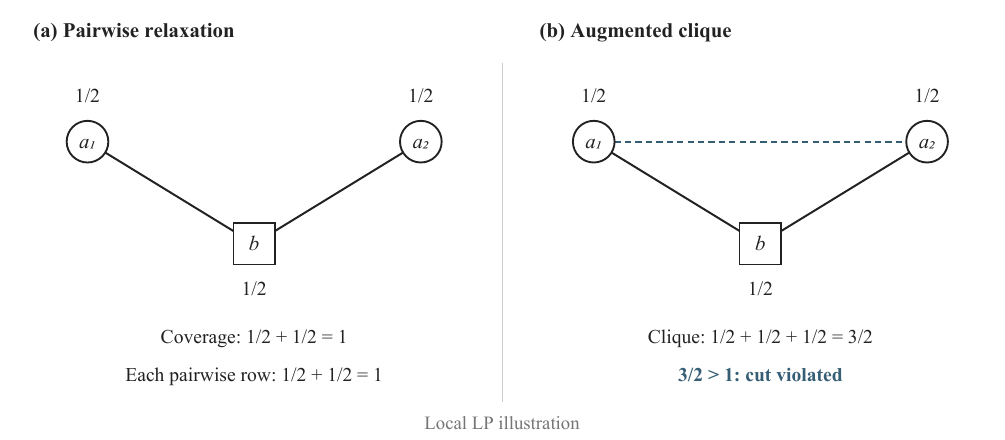}
 \caption{The three-vertex example: adding a same-request edge turns a physical path into a clique.}
 \label{fig:triangle-example}
\end{figure}

The graph $G^+$ may contain exponentially many cliques, making exhaustive enumeration impractical. Therefore, the separation procedure searches for violated clique inequalities of $G^+$ using greedy constructions guided by the current LP solution. The search may not identify every violated clique inequality. Each generated inequality is globally valid for the original problem, since at most one of its mutually incompatible moving arcs can be selected.

\subsection{Three clique generation methods}
\label{subsec:clique-methods}

We develop three complementary constructions for separating violated clique inequalities of $G^+$. The first exploits shared physical trajectories, the second combines physical conflicts with same-request incompatibilities, and the third extends a detected conflict to further departure alternatives of a given request and route. Let $\bar x$ denote the current LP solution.

\begin{definition}[Trajectory slot]
\label{def:slot}
The trajectory slot of a moving arc $a$ is
$\sigma(a)=(\rho_a,\tau_a)$. For each slot $s$, define
\[
  \Am_s=\{a\in\Am:\sigma(a)=s\},
  \qquad
  X_s=\sum_{a\in\Am_s}x_a.
\]
\end{definition}

Since the route option $\rho_a$ specifies the cruise altitude, all arcs in $\Am_s$ share the same departure time, arrival time, and physical trajectory. They may nevertheless represent service alternatives for different requests or a rebalancing flight. The quantity $X_s$ denotes their aggregate flow on slot $s$.

\paragraph{I. Physical slot cliques}
Distinct flights assigned to the same slot would coincide throughout their cruise intervals. The arcs in $\Am_s$ are therefore mutually incompatible, giving the same-slot inequality
\begin{equation}
  X_s=\sum_{a\in\Am_s}x_a\le1.
  \label{eq:same-slot-clique}
\end{equation}

More generally, consider a set $K$ of slots whose trajectories are pairwise conflicting. Physical incompatibility is independent of request labels, so every arc associated with one slot in $K$ conflicts with every arc associated with another. Together with the incompatibilities within each slot, this implies that $\bigcup_{s\in K}\Am_s$ is a clique of $G^+$. Hence, for a clique $K$, the associated inequality is:
\begin{equation}
  \sum_{s\in K}X_s
  =\sum_{s\in K}\sum_{a\in\Am_s}x_a\le1.
  \label{eq:physical-slot-clique}
\end{equation}
Candidate slot cliques are sought among slots with positive activity $\bar X_s=\sum_{a\in\Am_s}\bar x_a$. For each candidate $K$, the inequality includes all moving-arc variables associated with its slots, including those with zero current value.

\paragraph{II. Request-augmented cliques}
Alternatives for the same request are mutually exclusive even when their trajectories do not conflict. These logical incompatibilities can therefore complete cliques that would not be available from physical conflicts alone. We search for cliques on positive-valued moving arcs in $G^+$. For a clique $K^+$, the associated inequality is
\begin{equation}
  \sum_{a\in K^+}x_a\le1.
  \label{eq:request-augmented-clique}
\end{equation}
Its validity follows because each pair of members is excluded either by physical separation or by request coverage.

To complement the first construction, we focus on cliques containing both a physical conflict between arcs serving different requests (or involving a rebalancing arc) and a same-request incompatibility. Note that Method II serves as another search guidance that is not necessarily stronger than those generated by Method~I.

\paragraph{III. Request--route lifts}
A physical conflict can also be extended using the mutual exclusion of departure alternatives for a fixed request and route. Let $q$ be a request, let $a\in\Am\setminus\As_q$ be an outside moving arc, and
choose $\rho\in\Pp_{o_qd_q}$. Define
\begin{equation}
  L(a,q,\rho)
  =\{b\in\As_q:\rho_b=\rho,\ \{a,b\}\in\Ecol\}.
  \label{eq:route-lift-support}
\end{equation}
This set contains all feasible departure alternatives of request $q$ on route option $\rho$ that conflict physically with $a$. Every member of $L(a,q,\rho)$ is incompatible with $a$, and any two members are mutually exclusive because they serve the same request. Thus, whenever $L(a,q,\rho)$ is nonempty, $\{a\}\cup L(a,q,\rho)$ is a clique of $G^+$ and yields
\begin{equation}
  x_a+\sum_{b\in L(a,q,\rho)}x_b\le1.
  \label{eq:route-lift}
\end{equation}

A detected conflict between $a$ and an alternative of request $q$ using $\rho$ provides a starting point for this construction. The resulting set $L(a,q,\rho)$ is defined by the exact collision condition over all feasible departures, irrespective of their current LP values. Unlike the positive-support search in Method~II, this extension may therefore include zero-valued service alternatives.

Let $C$ contain all moving arcs appearing in a candidate clique inequality, including any zero-valued arcs included by slot expansion or request--route lifting. The inequality is considered for addition only if
\[
  \sum_{a\in C}\bar x_a>1+\eps^{\mathrm{sep}}.
\]
The comparison among three types of cliques are listed in Table \ref{tab:clique-generation-mechanisms}. The searching and selection of these violated inequalities is described in Section~\ref{subsec:separation-acceleration}.

\begin{table}[tbp]
\centering\small
\caption{The three clique generation methods.}
\label{tab:clique-generation-mechanisms}
\begin{tabularx}{\linewidth}{@{}>{\raggedright\arraybackslash}p{.18\linewidth}>{\raggedright\arraybackslash}X>{\raggedright\arraybackslash}X@{}}
\toprule
Method & Basis of the construction & Support of the inequality\\
\midrule
I. Physical slot
& One slot or several pairwise-conflicting slots
& All moving arcs in the selected slots\\
II. Request-augmented
& Physical and same-request incompatibilities among positive-valued arcs
& The moving arcs in the selected clique\\
III. Request--route
& One outside arc and the mutually exclusive alternatives of a request--route pair
& The outside arc and all departure alternatives of that pair that conflict with it\\
\bottomrule
\end{tabularx}
\end{table}

\subsection{Separation}
\label{subsec:separation-acceleration}

Given a node LP solution $\bar x$, separation searches for a candidate inequality whose violation
\begin{equation}
  v(C)=\sum_{a\in C}\bar x_a-1
  \label{eq:clique-violation}
\end{equation}
exceeds $\eps^{\mathrm{sep}}$, where $C$ contains the moving arcs in that inequality. The separator uses three modes: Root at the root node, Integer at nonroot nodes with integral moving decisions, and F16 at selected nonroot nodes with fractional moving decisions. F16 applies separation with frequency $F=16$ for every 16th eligible fractional node. Algorithm~\ref{alg:support-slot-rows} summarizes one separation call. Its three stages, conflict search, clique construction, and cut selection, are described below.

\begin{algorithm}[t]
\small
\caption{Separation}
\label{alg:support-slot-rows}
\label{alg:slot-separation}
\begin{algorithmic}[1]
\Require Node LP solution $\bar x$, active inequalities $\mathcal R$,
and mode $\kappa\in\{\mathrm{Root},\mathrm{Integer},\mathrm{F16}\}$
\Ensure Additions and dominated removals $\Delta$ to $\mathcal R$
\Statex \textbf{Presets:}
\[
  (C_\kappa,B_\kappa)=
  \begin{cases}
    (100000,\infty), & \kappa=\mathrm{Root},\\
    (1000,\infty),   & \kappa=\mathrm{Integer},\\
    (2000,80),       & \kappa=\mathrm{F16},
  \end{cases}
  \qquad
  \eps^+=\eps^{\mathrm{sep}}=10^{-7}.
\]
\Function{SupportSlotRows}{$\bar x,\mathcal R,\kappa$}
  \Statex \textbf{Conflict search}
  \State $\mathcal A^+\gets\{a\in\Am:\bar x_a>\eps^+\}$
  \State $\mathcal S^+\gets\{\sigma(a):a\in\mathcal A^+\}$
  \State $H\gets\Call{ConflictSearch}{\mathcal S^+,C_\kappa}$
  \Statex \textbf{Clique construction}
  \State $\Gamma\gets\Call{CliqueConstruction}{\mathcal A^+,H,\bar x}$
    \Comment{constructions I--III}
  \Statex \textbf{Cut selection}
  \State $\Gamma\gets\{C\in\Gamma:v(C)>\eps^{\mathrm{sep}}\}$
  \State Remove duplicate or previously represented candidates from $\Gamma$
  \If{$B_\kappa<\infty$}
    \State Retain at most $B_\kappa$ candidates with largest $v(C)$
  \EndIf
  \State Apply support dominance to $\Gamma$ and $\mathcal R$;
    record additions and removals in $\Delta$
  \State \Return $\Delta$
\EndFunction
\end{algorithmic}
\end{algorithm}

\paragraph{Conflict search}
The search begins with positive-valued moving arcs and groups them by slot, thereby avoiding repeated geometric comparisons for arcs with different request labels but identical trajectories. Pairs at different altitudes, with disjoint cruise intervals, or with sufficiently separated bounding boxes are discarded. The remaining pairs are checked using the exact collision condition, with previous results reused for repeated trajectory comparisons. The graph $H$ contains all positive slots and the confirmed conflict edges. The search terminates once $C_\kappa$ conflicts between distinct slots have been found. This cap counts conflicts rather than pair comparisons and may leave additional conflicts undiscovered.

\paragraph{Clique construction}
For construction~I, a greedy search starts from an uncovered edge of $H$ and repeatedly adds a common neighbour of all current clique members. Among eligible vertices, it favours the one with the most neighbours in the remaining candidate set. When no extension is possible, the clique's edges are marked as covered, and another uncovered edge is selected. Each positive slot is also tested on its own using \eqref{eq:same-slot-clique}. Then, each selected slot clique is expanded to all moving arcs in its slots.

Construction~II reuses the slot conflicts recorded in $H$, so building its graph requires no additional geometric collision checks. Its vertices are positive-valued moving arcs. Two arcs are connected if they occupy the same slot, occupy slots joined by an edge of $H$, or serve the same request. The same greedy procedure is then applied to this augmented graph.

Construction~III considers each pair $(a,b)$ of positive-valued moving arcs that occupy the same slot or two slots identified as physically conflicting. If $b=(q,\rho_b,\tau_b)$ is a service arc and $a\notin\As_q$, then $b$ supplies the target request $q$ and route $\rho_b$. The outside arc $a$ serves another request or is a rebalancing arc. Keeping $a$, $q$, and $\rho_b$ fixed, all feasible departure alternatives of request $q$ on $\rho_b$ are checked for physical conflicts with $a$. This gives $L(a,q,\rho_b)$ and the candidate inequality~\eqref{eq:route-lift}. When both arcs serve different requests, both orientations are considered, with each arc acting in turn as the outside arc. The additional tests required to form a lift are separate from the limit $C_\kappa$ on the slot search.

The inequalities obtained from constructions~I and~III include all arcs specified by their definitions, including those with zero current value. The test in \eqref{eq:clique-violation} is then applied to each candidate. These searches may not identify all violated clique inequalities. However, each generated inequality is valid because no two arcs in the corresponding clique can be selected together.

\paragraph{Cut selection}
The separator discards repeated inequalities already selected or replaced by stronger ones. When $B_\kappa$ is finite, the remaining candidates are ranked by violation and at most $B_\kappa$ are selected across the three constructions. Candidates omitted because they are unviolated or exceed this budget can be reconsidered at a later LP solution.

Dominated inequalities are removed when the selected cuts are added.
For $C_1\subsetneq C_2$, nonnegativity gives
\begin{equation}
  \sum_{a\in C_2}x_a\le1
  \quad\Longrightarrow\quad
  \sum_{a\in C_1}x_a\le1.
  \label{eq:row-dominance}
\end{equation}
Thus, a candidate contained in an active inequality is discarded, while a newly found inequality can replace weaker ones whose variable sets it contains. These comparisons apply across all three constructions and may reduce the number of cuts actually added.

\subsection{Branching and search management}
\label{subsec:branching}

\paragraph{Node selection}
Before an incumbent is available, we explore nodes by a depth-first dive. We then order open nodes by $(\underline z_N,\operatorname{depth}(N),\operatorname{sequence}(N))$, where $N$ is a branch-and-cut node. We prefer lower $\underline z_N$. Among nodes with the same $\underline z_N$, we select the shallowest node. If several nodes have the same depth, we select the one that entered the queue first.

\paragraph{Branching}
We apply the following three branching rules in order. For request $q$, let $\As_q(N)$ be its service alternatives still allowed at node $N$. A generalized upper-bound (GUB) branch divides this set into a nonempty group $L$ and its complement. The two children impose
\begin{equation}
  \sum_{a\in L}x_a=1
  \qquad\text{or}\qquad
  \sum_{a\in L}x_a=0.
  \label{eq:gub-branch}
\end{equation}

\begin{enumerate}
\item \textbf{Congestion-based request GUB.}
Let $C$ be the moving-arc set of an active clique inequality. For each request $q$, set $L=C\cap\As_q(N)$. Define
\begin{equation}
  m=\sum_{a\in L}x_{a},\qquad
  o=\sum_{a\in C\setminus\As_q}x_{N,a},\qquad
  \operatorname{score}(C,q)=\min\{m,1-m\}\,o.
  \label{eq:congestion-score}
\end{equation}
We select the clique $C$ and request $q$ with the highest score, subject to $0<m<1$ and $o>0$. The score favours splits with balanced LP weights, measured by \(\min\{m,1-m\}\), and substantial LP weight on conflicting arcs outside request \(q\), measured by \(o\).

\item \textbf{Balanced request GUB.}
If no congestion-based split is available, consider the set $\mathcal Q_f(N)$ of requests whose service decisions remain fractional. For each $q\in\mathcal Q_f(N)$, process its positive-valued alternatives in decreasing LP value, assigning each to the lighter of two initially empty groups, where weight is the sum of LP values. Denote the resulting groups by $B_{q,1}$ and $B_{q,2}$, and define
\begin{equation}
  L_q=B_{q,1},
  \qquad
  m_q=\sum_{a\in L_q}x_{a}.
  \label{eq:balanced-gub-partition}
\end{equation}
All allowed zero-valued alternatives remain in the complement: $  \As_q(N)\setminus L_q=B_{q,2}\cup\{a\in\As_q(N):x_{a}=0\}$.
Among these candidate partitions, select
\begin{equation}
  q^\star\in
  \arg\max_{q\in\mathcal Q_f(N)}
  \min\{m_q,1-m_q\},
  \qquad
  L=L_{q^\star}.
  \label{eq:balanced-gub-choice}
\end{equation}

\item \textbf{Rebalancing branching K1.}
If all service decisions are integral, we choose the fractional
rebalancing variable closest to $1/2$ and create the two children
\begin{equation}
  x_a=0\qquad\text{or}\qquad x_a=1,
  \quad a\in\Ar.
  \label{eq:rebal-branch}
\end{equation}
\end{enumerate}

\paragraph{Primal heuristic}
When the search tree finds a feasible integer solution, the primal heuristic improves it by adjusting departure times, routes, cruising altitudes, and rebalancing flights while strictly preserving each drone’s request sequence in the planning horizon.

Algorithm~\ref{alg:explicit-lp-bc} applies branch-and-cut to the model \eqref{model:exact}, adding collision inequalities as needed. Let $N_0$ be the root, $\mathcal S$ the open-node set, $\mathcal R$ the active global inequalities, and $(\hat x,\bar z)$ the best verified incumbent and its objective value. We use $P=24$ as the maximum number of separation iterations for the root node, nonroot separation frequency $F=16$, wall-time limit $W$, and the tolerances defined above. The counter $r$ counts root separation iterations. The counter $\nu$ counts nonroot nodes when they first have an unpruned fractional solution. For node $N$, $u_N$ indicates whether its scheduled fractional separation F16 has been used. The sets $\Delta^+$ and $\Delta^-$ contain added and removed inequalities.

\begin{algorithm}[H]
\small
\raggedright
\algrenewcommand\algorithmicthen{}
\algtext*{EndIf}
\caption{Branch-and-cut: main search procedure}
\label{alg:explicit-lp-bc}
\begin{algorithmic}[1]
\State Initialize $\mathcal S\gets\{N_0\}$,
$\underline z_{N_0}\gets-\infty$, $\bar z\gets+\infty$,
$\hat x\gets\varnothing$, and $\nu\gets0$
\While{$\mathcal S\ne\varnothing$ and elapsed time $<W$}
  \State Pop $N$ by depth-first dive before the first incumbent; Otherwise by $(\underline z_N,\operatorname{depth}(N),\operatorname{sequence}(N))$
  \State Set $r\gets0$, and $u_N\gets\mathrm{false}$;
Solve the node LP and, if feasible, obtain $(z_N,x_N)$
  \Loop
    \If{the LP is infeasible} or {$z_N\ge\bar z-\eps^{\mathrm{obj}}$}
      \State Fathom $N$; \textbf{break}
    \EndIf
    \If{$N$ is the root}
      \State $r\gets r+1$
    \EndIf
    \If{$N$ is nonroot, $x_N$ is fractional, and $\neg u_N$}
      \State $\nu\gets\nu+1$;
    \EndIf
    \State $\displaystyle\kappa\gets
      \begin{cases}
        \mathrm{Root}, & N\text{ root},\quad (x_N\text{ integral}\;\lor\;r \leq P),\\
        \mathrm{Integer}, & N\text{ nonroot},\quad x_N\text{ integral},\\
        \mathrm{F16}, & N\text{ nonroot},\;x_N\text{ fractional},\; \neg u_N,\;
(\nu-1)\bmod F=0,\\
        \varnothing, & \text{otherwise}.
      \end{cases}$
    \If{$\kappa=\varnothing$}
      \State \textbf{break}
    \EndIf
    \If{$\kappa=\mathrm{F16}$}
      \State $u_N\gets\mathrm{true}$
    \EndIf
    \State $\Delta\gets\Call{Algorithm 1: Separation}{x_N,\mathcal R,\kappa}$
    \If{$\Delta^+=\varnothing$}
      \State \textbf{break}
      \Comment{the current solution is closed under mode $\kappa$}
    \EndIf
    \State $\mathcal R\gets(\mathcal R\setminus\Delta^-)\cup\Delta^+$

    \State Reoptimize the node LP to obtain $(z_N,x_N)$
    
  \EndLoop
  \If{$N$ was fathomed}
  \State \textbf{continue} to the next open node
  \EndIf
    \If{$x_N$ is integral and feasible}
  \State try to improve $x_N$ by primal heuristic
  \State save it as a projected incumbent if $z_N<\bar z$
\Else
  \State branch by congestion GUB, balanced request GUB,
  or rebalancing K1
\EndIf
\EndWhile
\end{algorithmic}
\end{algorithm}

\section{Experiments}
\label{sec:experiments}
Section \ref{sub: data preparation} displays the common settings in the following two experiments. Section \ref{subsec: computational results} conducts the numerical experiment that compares the proposed algorithm with Gurobi's native lazy-constraint functionality. Section \ref{subsec: operational experiments} conducts the operational experiment that investigates the influence of multiple route candidates per vertiport pair on noise exposure, operation cost, and collisions.

\subsection{Data and network preparation}
\label{sub: data preparation}

\begin{figure}[h!]
\centering
\includegraphics[width=\linewidth,height=\textheight,keepaspectratio]{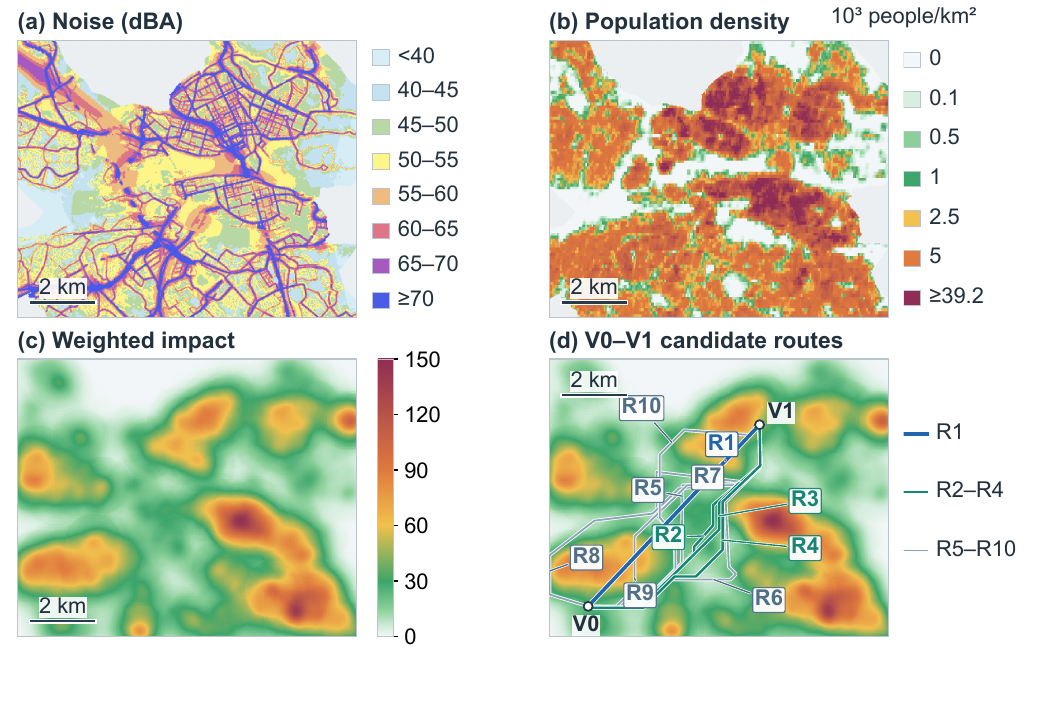}
\caption{Spatial data inputs and nested candidate planar routes.}
\label{fig:raster-data-route-generation}
\end{figure}

The networks are constructed using the European Commission Joint Research Centre's GHS-POP R2023A population data for 2020 (WGS84, 3 arc-second resolution) \citep{ghspop,ghsl2024} and the City of Stockholm Environmental Administration's \textit{Bullerkartan 2022} \citep{stockholmnoise}. The latter provides modeled 24-hour equivalent noise levels from road, rail, and aircraft traffic. During preprocessing, the noise polygons are rasterized onto a 25~m grid, as shown in Fig.~\ref{fig:raster-data-route-generation}(a). Then, population counts are sampled on the same grid and converted to population densities, as shown in Fig.~\ref{fig:raster-data-route-generation}(b). These two inputs are combined through the noise-impact model in \ref{app:noise-cost}, specifically
Eq.~\eqref{eq:noise-population-impact}, to obtain the population-weighted impact raster in Fig.~\ref{fig:raster-data-route-generation}(c).

Vertiports are first generated and should be located at least 1.6~km apart. Candidate planar routes are generated offline for each unordered vertiport pair based on the 25~m weighted impact. Edge weights combine flight distance with population-weighted noise exposure evaluated at a reference altitude of 300~m. The generator considers a direct connection where noise-data coverage permits, a minimum-cost path, and alternative paths through laterally displaced intermediate nodes. Route lengths are limited to 2.4 times the straight-line distance between the endpoints, and an initial edge-overlap filter encourages spatial diversity. Candidates are simplified into polylines with at most eight straight-line segments using deviation tolerances of up to 350~m. At the same time, the retained geometries must be within the available noise-data area.

Fig.~\ref{fig:raster-data-route-generation}(d) displays an example of generating ten candidate routes between two vertiports $v_0$ and $v_1$ on the same impact raster. Let $K$ denote the number of retained planar routes per unordered vertiport pair. Labels R1--R10 indicate the fixed order of addition: R1 is always the direct route between two vertiports at $K=1$, and R2--R10 are added subsequently. Increasing $K$ thus provides additional spatial alternatives while retaining all previously available routes. Each retained geometry is evaluated at both 300 and 500~m to obtain its altitude-specific exposure and cost and is made available in both directions. Expansion by direction and altitude therefore yields $2K$ route--altitude options per directed OD pair. The cost of each moving arc in the objective function \eqref{eq:objective} is illustrated in \ref{app:noise-cost}.

Demand is generated from a population-based gravity distribution. There are 60 or 100 transportation requests, and request release times are drawn uniformly from the integers $0,\ldots,45$. The planning horizon comprises 90 time intervals, with time points $0,\ldots,90$. All requests must be served within the planning horizon. Each vertiport initially holds two or three aircraft and must recover the same inventory by the end of the planning horizon.

\subsection{Numerical results}
\label{subsec: computational results}
Explicitly including all pairwise collision constraints creates a substantial computational burden. In preliminary tests, this direct implementation exceeded the available memory on the test machine. Therefore, we compare the proposed solution approach, named B123, against the Gurobi solver with a lazy-constraint mechanism on the same 40 instances, denoted by Lazy. Both methods generate collision constraints on demand and are evaluated on the same 40 instances.

\subsubsection{Experimental settings}
In generating the 40 instances, we consider the following four factors: the number of requests $R\in\{60,100\}$, the number of vertiports $D\in\{8,10\}$, the number of vehicles initially available at each vertiport $F\in\{2,3\}$, and the number of planar route alternatives per vertiport pair $K\in\{4,8\}$. The total fleet size is $V=DF$. In the order $(R,D,V,K)$, the eight configurations are $(60,8,16,4)$, $(60,8,24,8)$, $(60,10,20,8)$, $(60,10,30,4)$, $(100,8,16,8)$, $(100,8,24,4)$, $(100,10,20,4)$, and $(100,10,30,8)$. Each configuration has five different instances. Each instance regenerates vertiports and requests. The instance name \texttt{R60\_D8\_V16\_K4\_1}, for example, represents the first replicate with the configuration of 60 requests, 8 vertiports, 16 vehicles, and 4 planar route alternatives per vertiport pair. All instances require a $\Rsep =180\ \text{m}$ safety distance during cruise stages for all flights.

The implementations use Python 3.11.4 and Gurobi 11.0.1 on a Windows laptop with an Intel Core i9-12900H CPU and 32 GB of RAM. Runs are executed sequentially, with a nominal time limit of 3,600 s. Each run starts with no externally supplied initial solution and a separate empty conflict cache, which is populated on demand. Note that B123 processes its custom branch-and-cut tree sequentially in Python and solves the node LP relaxations using Gurobi's single-threaded dual simplex algorithm. Lazy uses Gurobi's native parallel MIP solver, including its built-in presolve, cutting planes, and primal heuristics; the solver logs report the use of 20 threads in all 40 runs. The runtime comparison therefore reflects both the underlying search strategies and differences in implementation and parallelism between our B123 and Gurobi's Lazy.

\subsubsection{Comparison with native MIP search and lazy collision constraints}
\label{sec:comparison-t90}

We compare B123 with Lazy, which uses Gurobi's native MIP search with the native lazy-constraint function for pairwise collision constraints \eqref{eq:pairwise-collision}. Table~\ref{tab:t90-summary} summarizes the results, and Table~\ref{tab:t90-full-results} reports all paired observations.

\begin{table}[h!]
\centering
\caption{Comparison of B123 and Lazy by instance configuration.}
\label{tab:t90-summary}
\begingroup\fontsize{8.5}{10.2}\selectfont\setlength{\tabcolsep}{3pt}
\textbf{(a) Feasibility, optimality, and bounds}\par\smallskip
\begin{tabular}{@{}clrrrrrrrrrr@{}}
\toprule
& & \multicolumn{2}{c}{Feasible} & \multicolumn{2}{c}{Optimal} & \multicolumn{3}{c}{Higher LB} & \multicolumn{3}{c}{Lower UB} \\
\cmidrule(lr){3-4}\cmidrule(lr){5-6}\cmidrule(lr){7-9}\cmidrule(l){10-12}
Case & $(R,D,V,K)$ & B123 & Lazy & B123 & Lazy & B123 & Equal & Lazy & B123 & Equal & Lazy \\
\midrule
1 & (60,8,16,4) & 5/5 & 5/5 & 5/5 & 5/5 & 0 & 5 & 0 & 0 & 5 & 0 \\
2 & (60,8,24,8) & 5/5 & 5/5 & 4/5 & 4/5 & \textbf{1} & 4 & 0 & \textbf{1} & 4 & 0 \\
3 & (60,10,20,8) & 5/5 & 5/5 & 5/5 & 5/5 & 0 & 5 & 0 & 0 & 5 & 0 \\
4 & (60,10,30,4) & 5/5 & 5/5 & 5/5 & 5/5 & 0 & 5 & 0 & 0 & 5 & 0 \\
5 & (100,8,16,8) & \textbf{5/5} & 0/5 & 0/5 & 0/5 & \textbf{5} & 0 & 0 & -- & -- & -- \\
6 & (100,8,24,4) & \textbf{5/5} & 0/5 & 0/5 & 0/5 & \textbf{5} & 0 & 0 & -- & -- & -- \\
7 & (100,10,20,4) & \textbf{5/5} & 2/5 & 1/5 & \textbf{2/5} & \textbf{3} & 0 & 2 & 0 & 1 & \textbf{1} \\
8 & (100,10,30,8) & \textbf{5/5} & 0/5 & 0/5 & 0/5 & \textbf{5} & 0 & 0 & -- & -- & -- \\
\midrule
All &  & \textbf{40/40} & 22/40 & 20/40 & \textbf{21/40} & \textbf{19} & 19 & 2 & 1 & 20 & 1 \\
\bottomrule
\end{tabular}\par\medskip
\textbf{(b) Average computation time and search counters}\par\smallskip
\begin{tabular}{@{}clrrrrrr@{}}
\toprule
& & \multicolumn{2}{c}{Average time (s)} & \multicolumn{2}{c}{Average cuts} & \multicolumn{2}{c}{Average nodes} \\
\cmidrule(lr){3-4}\cmidrule(lr){5-6}\cmidrule(l){7-8}
Case & $(R,D,V,K)$ & B123 & Lazy & B123 & Lazy & B123 & Lazy \\
\midrule
1 & (60,8,16,4) & 458.2 & \textbf{97.0} & \textbf{1,869.2} & 2,610.4 & \textbf{1,019.2} & 20,576.2 \\
2 & (60,8,24,8) & 1,966.2 & \textbf{963.8} & \textbf{3,409.4} & 4,952.6 & \textbf{3,053.6} & 52,311.0 \\
3 & (60,10,20,8) & 664.8 & \textbf{178.8} & \textbf{2,034.4} & 2,610.2 & \textbf{956.8} & 12,776.4 \\
4 & (60,10,30,4) & 66.6 & \textbf{40.2} & \textbf{1,041.4} & 1,223.2 & \textbf{215.6} & 6,847.4 \\
5 & (100,8,16,8) & 3,600.3 & 3,606.2 & \textbf{5,419.0} & 10,185.0 & \textbf{862.0} & 56,159.8 \\
6 & (100,8,24,4) & 3,600.3 & 3,604.2 & \textbf{7,771.0} & 9,537.4 & \textbf{1,367.4} & 102,652.4 \\
7 & (100,10,20,4) & 3,502.6 & \textbf{2,979.4} & \textbf{5,611.8} & 6,761.0 & \textbf{1,431.4} & 75,339.2 \\
8 & (100,10,30,8) & 3,600.2 & 3,606.7 & \textbf{6,768.8} & 9,261.0 & \textbf{941.4} & 98,327.8 \\
\midrule
All &  & 2,182.4 & \textbf{1,884.5} & \textbf{4,240.6} & 5,892.6 & \textbf{1,230.9} & 53,123.8 \\
\bottomrule
\end{tabular}\par\endgroup
\end{table}

In Table \ref{tab:t90-summary}, each configuration contains five paired instances. Higher final LB and Lower final UB count B123 wins, ties (Equal), and Lazy wins. Bold in panel~(a) marks larger feasible/optimal counts or bound-win counts between methods. Bold in panel~(b) marks lower average computation time/cuts/nodes.

In Table~\ref{tab:t90-summary}(a), the main difference is feasible-solution coverage: B123 finds schedules for all 40 instances, whereas Lazy does so for 22. All 18 additional feasible schedules occur in the tested $R=100$ configurations (Cases~5--8), suggesting that B123 has a better ability to find an incumbent. At the same time, B123 obtains higher final LBs on 19 instances, equal LBs on 19, and lower LBs on two, showing generally better performance on improving the LBs.

Table~\ref{tab:t90-summary}(b) shows that B123 has lower mean cut and node counts in every configuration. Across all 40 instances, the mean recorded node counts are 1,230.9 for B123 and 53,123.8 for Lazy. Thus, Lazy processes approximately 43.16 times as many nodes as B123, with the ratio calculated from the unrounded counts. The mean recorded search times are 2,182.4~s for B123 and 1,884.5~s for Lazy. These averages include all 40 runs per method, including time-limited runs, runs without an incumbent, and recorded overruns. As discussed in the experimental settings, B123 manages its custom branch-and-cut tree sequentially in Python, whereas Lazy uses Gurobi's native parallel MIP search with 20 threads. The longer mean runtime of B123 should be interpreted together with the factor of Lazy's parallel search and other Gurobi's native acceleration techniques.

Overall, on the tested configurations, B123 provides complete feasible-solution coverage and generally higher lower bounds while processing substantially fewer search nodes. These combined results demonstrate the effectiveness of the proposed framework in obtaining feasible schedules and strong bounds with limited node exploration. Lazy proves optimality faster on the instances solved to optimality by both methods, indicating scope for improving the computational speed of the B123 implementation.

\subsection{Route flexibility, noise exposure and cruise separation}
\label{subsec: operational experiments}

\begin{figure}[h!]
\centering
\includegraphics[width=\linewidth,height=0.65\textheight,keepaspectratio]{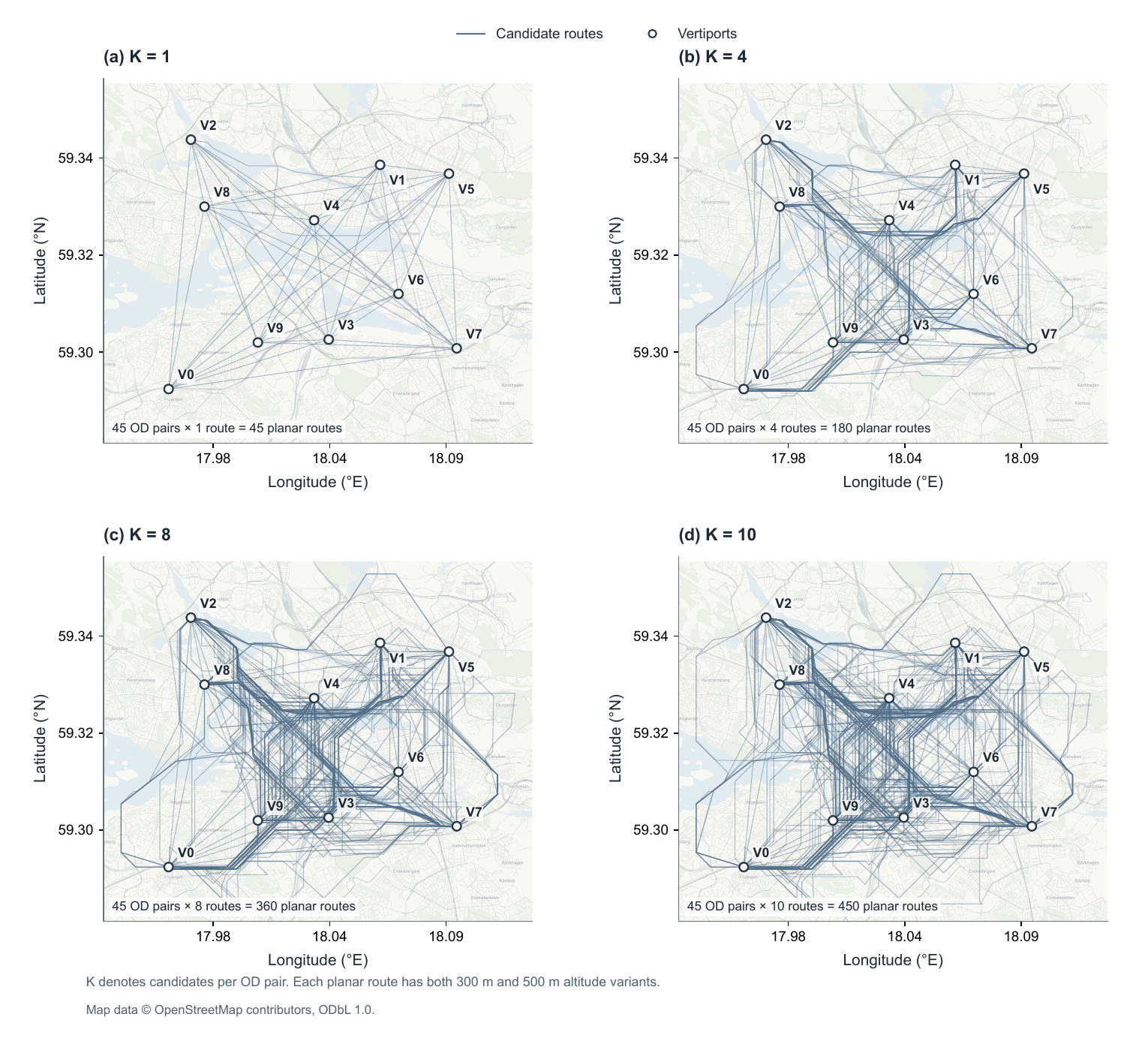}
\caption{Candidate route networks for $K=1,4,8,10$.}
\label{fig:m1m0-network}
\end{figure}

Additional routes give aircraft more ways to reduce noise exposure, but these alternatives must also fit into a feasible fleet schedule. This experiment examines how much route flexibility improves the noise exposure and operating cost, and how that benefit changes when cruise separation is enforced.

\subsubsection{Experimental design}
We compare two models on one of the largest configurations \texttt{R100\_D10\_V30\_K8\_1} in Section \ref{subsec: computational results}: M1 from model \eqref{model:exact}, whereas M0 relaxes the cruise-separation constraints of 180m safety distance \eqref{eq:pairwise-collision}. M0 therefore provides a lower-cost relaxation against which to measure the effect of this requirement.

We vary the number of candidate planar routes per vertiport pair, $K$, from 1 to 10. Fig.~\ref{fig:m1m0-network} shows the candidate networks at $K=1,4,8,10$. Across the 45 unordered vertiport pairs, these networks contain 45, 180, 360, and 450 planar routes, respectively. The catalogue is constructed as a fixed nested sequence. At $K=1$, each pair retains the shortest route: a straight route for 43 pairs and a bent route for the remaining two due to the restriction of the feasible area. The other candidates are added in increasing order of length. Therefore, a feasible schedule at one value of $K$ remains available at every larger value. The comparison measures the value of this particular sequence of alternatives, whose geometry changes along with the size of the catalogue. 

We therefore consider two models, M0 and M1, and ten route catalogues with $K=1,\ldots,10$ candidate planar routes per OD. The 20 model--catalogue combinations are evaluated using Algorithm \ref{alg:explicit-lp-bc}. A validated solution from a smaller catalogue can initialize the same model at a larger $K$. All twenty configurations are solved to optimality with the help of an initial feasible solution and a global warm cache of collisions across all runs.

\subsubsection{Value of additional route candidates}

\begin{figure}[H]
\centering
\includegraphics[width=\linewidth,height=0.4\textheight,keepaspectratio]{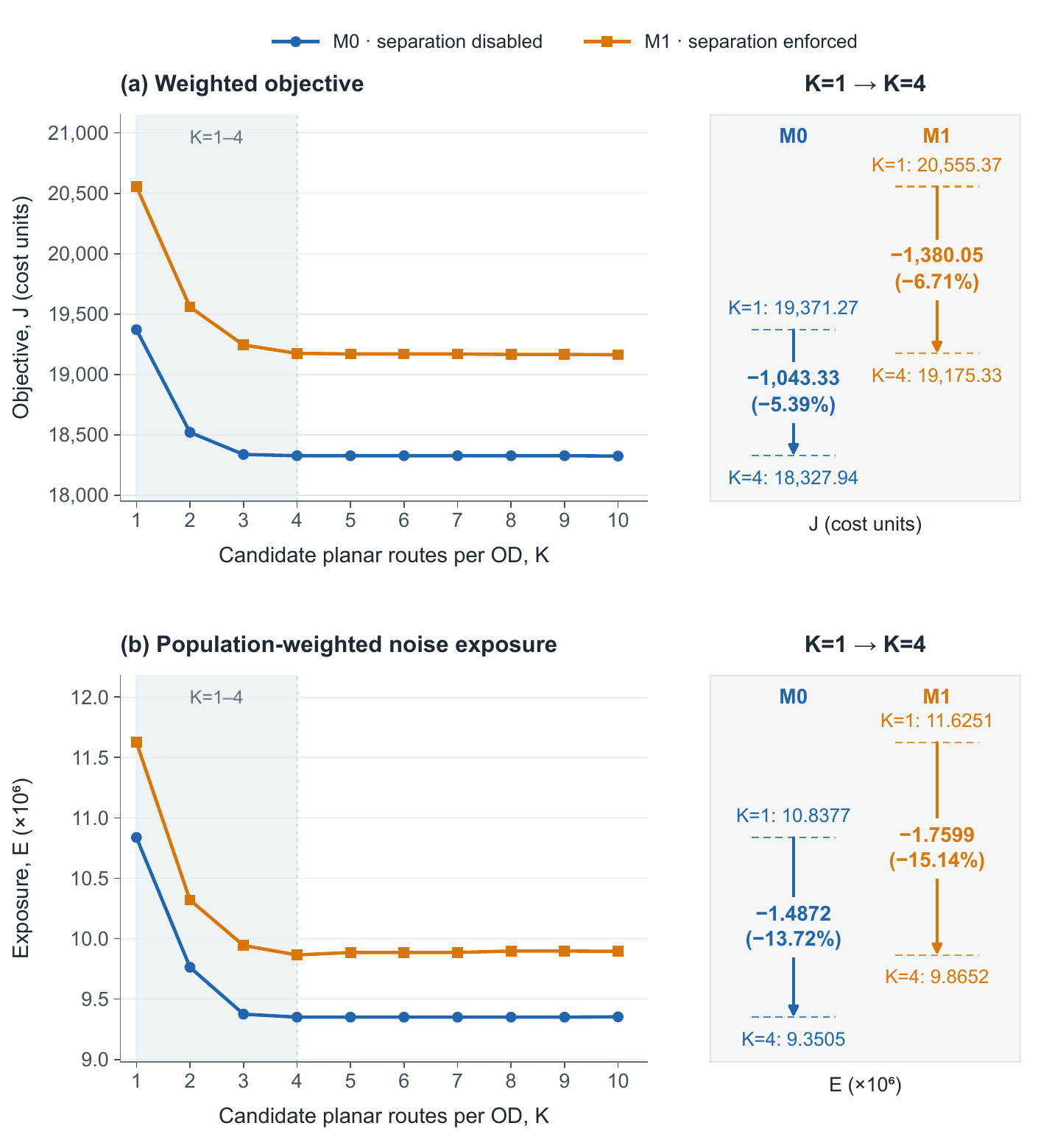}
\caption{Weighted objective and population-weighted noise exposure as the route candidates expand.}
\label{fig:m1m0-objective}
\end{figure}

Fig.~\ref{fig:m1m0-objective} shows the objective value and noise exposure in panels (a) and (b), respectively, as the number of candidate planar routes per vertiport pair, $K$, increases. Panel (a) shows that the main gains in the objective occur between $K=1$ and $K=4$. Under M1, the objective falls 6.71\%, while exposure falls 15.14\%. M0 follows the same broad pattern, with reductions of 5.39\% in the objective and 13.72\% in exposure. The larger reductions under M1 suggest that additional routes are more useful when flights must also be separated. Beyond four routes, the objective curves are nearly flat, showing that almost all of the objective reduction is obtained through the first few additions. Panel (b) shows a similar concentration of exposure reductions at low values of $K$.

\begin{figure}[H]
\centering
\includegraphics[width=\linewidth,height=0.68\textheight,keepaspectratio]{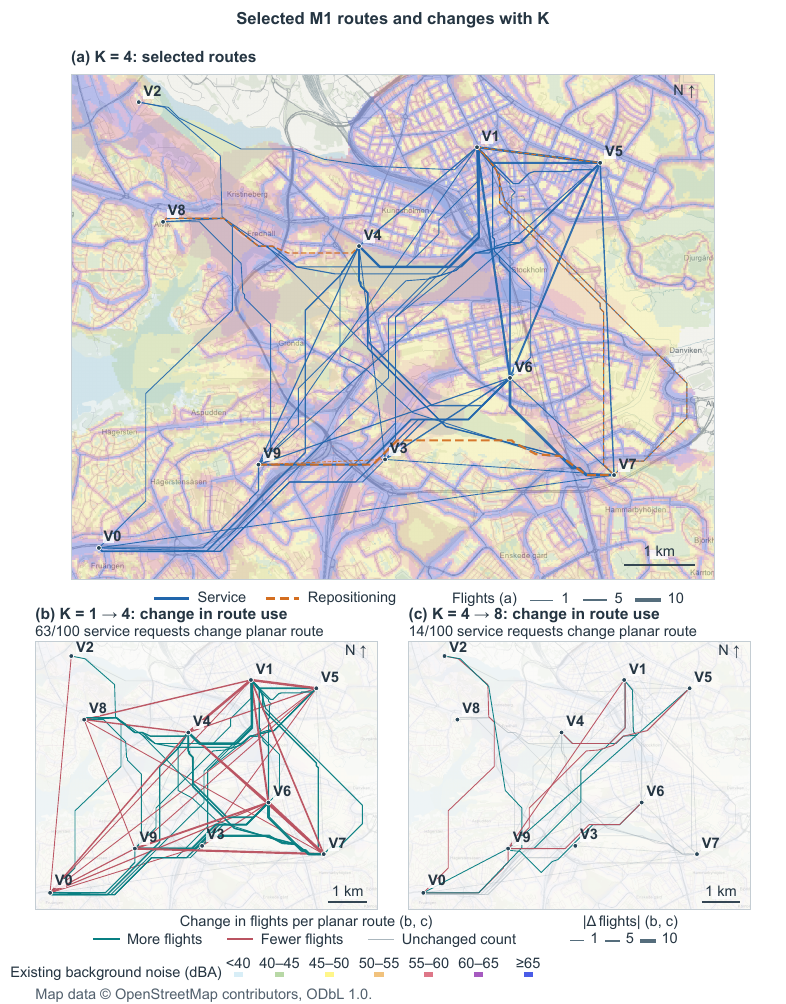}
\caption{Selected M1 routes and changes in route use. $K=1$: 100 service + 16 rebalancing flights; 33 distinct planar routes; 412.51~km total planar distance. $K=4$: 100 service + 15 rebalancing flights; 46 distinct planar routes; 457.54~km total planar distance. $K=8$: 100 service + 15 rebalancing flights; 48 distinct planar routes; 459.87~km total planar distance.}
\label{fig:m1m0-routes}
\end{figure}

The spatial changes in Fig.~\ref{fig:m1m0-routes} help explain why the early expansion matters more. Panel (a) shows the M1 schedule at $K=4$: 100 service flights and 15 repositioning flights use 46 distinct planar routes, compared with 33 at $K=1$. Between these two schedules, 63 of the 100 requests change planar route. The widespread increases and decreases in route use in panel (b) correspond to this redistribution. Total planar distance across service and repositioning flights rises from 412.51 to 457.54 km, an increase of 10.92\%. Thus, the early benefit comes with substantial rerouting and additional distance. The subsequent expansion produces much smaller changes. From $K=4$ to $K=8$, only 14 service requests change planar route, and the number of used planar routes increases from 46 to 48. Total planar distance increases from 457.54 to 459.87 km by a further 0.51\%. Panel (c) therefore depicts local adjustments around a largely retained route pattern, consistent with the small additional objective gains.

\begin{figure}[h!]
\centering
\includegraphics[width=0.6\linewidth,height=0.76\textheight,keepaspectratio]{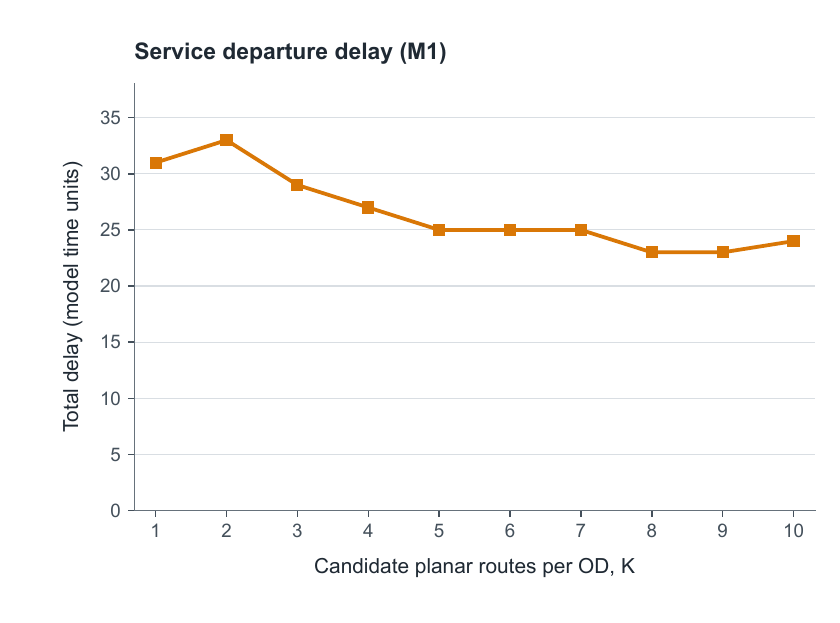}
\caption{Service departure delay in the reported M1 schedules.}
\label{fig:m1m0-delay}
\end{figure}

Fig.~\ref{fig:m1m0-delay} reports total service departure delay in the M1 schedules, calculated as the sum of the release-to-departure times across all 100 requests. The lowest observed delay is 23 model time units at $K=8$ and $K=9$, compared with 31 at $K=1$, representing a 25.81\% reduction. Half of this net reduction is observed between $K=4$ and $K=8$, where delay decreases from 27 to 23, indicating that improvements in service timing extend beyond the initial catalogue expansion. The reported $K=10$ schedule has a slightly higher total delay of 24 model time units.

\begin{figure}[h!]
    \centering
    \includegraphics[width=\textwidth]
    {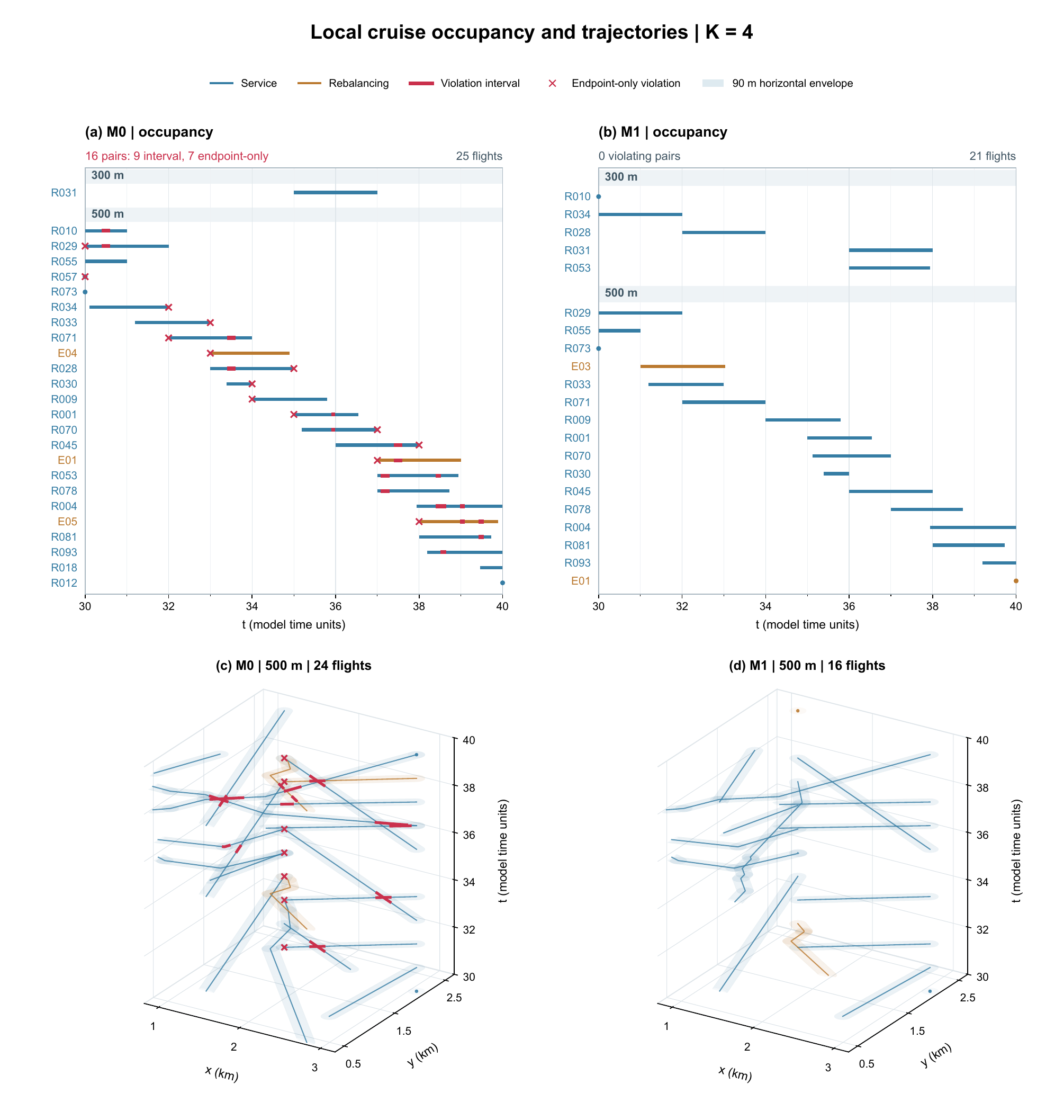}
    \caption{Local cruise occupancy and separation status in the reported M0 and M1 schedules at $K=4$.}
    \label{fig:k4-occupancy-trajectories}
\end{figure}

Fig.~\ref{fig:k4-occupancy-trajectories} gives an example of collision avoidance by comparing local cruise operations in the reported M0 and M1 schedules at $K=4$. Both schedules are examined in the same spatial region over $t\in[30,40]$ for illustrative purposes. Panels (a) and (b) show the cruise occupancy of all flights passing through the displayed neighborhood, grouped by altitude. Continuous trajectory checks identify 16 local violating pairs under M0, including nine with violations over positive-duration intervals and seven confined to a shared cruise endpoint, whereas no local violations are found under M1. Panels (c) and (d) show the corresponding 500~m space--time trajectories, with time on the vertical axis and 90~m horizontal envelopes used to visualize the 180~m pairwise separation requirement. Several M1 occupancy intervals overlap without violating separation, illustrating that concurrent use of the same local airspace is permitted provided that same-altitude aircraft remain sufficiently separated at each instant. Thus, overlapping occupancy intervals or projected trajectory crossings alone do not constitute a separation violation. The full-schedule post-processing check likewise finds no cruise-separation violations in the reported M1 schedule.

\section{Conclusion}
\label{sec: conclusion}
This study developed an integrated routing, scheduling, and rebalancing framework for noise-aware on-demand urban air mobility with multiple candidate flight paths per vertiport pair.
A time-expanded fleet-flow formulation jointly selects paths from these candidate sets for service and repositioning flights, together with cruise altitudes, departure times, and empty repositioning movements. The formulation accounts for population-weighted noise exposure and cruise separation while tracking the availability of a finite aircraft fleet. The proposed branch-and-cut algorithm exploits physical flight conflicts and request exclusivity through three complementary clique constructions. Across 40 Stockholm-based instances, the proposed branch-and-cut algorithm finds feasible solutions for all 40 instances, compared with 22 for Lazy, and achieves stronger lower-bound performance overall with lower mean cut and node counts. These results are obtained with lower mean cut counts and substantially fewer search nodes: B123 processes 1,230.9 nodes on average, compared with 53,123.8 for Lazy, which processes approximately 43.16 times as many nodes.

The operational case study shows that a small set of spatially differentiated routes per vertiport pair can capture most of the objective improvement available from the larger tested catalogue. Under cruise separation, expanding from one to four alternatives substantially reduces population-weighted exposure by 15.14\%, while increasing total planar flight distance from 412.51 to 457.54~km and changing the routes of 63 of the 100 service requests. Lower exposure therefore comes partly through accepting longer flights. Expanding from four to eight alternatives produces much smaller spatial adjustments: only 14 of the 100 requests change route, total planar distance increases from 457.54 to 459.87~km, and population-weighted exposure remains almost the same. Nevertheless, total service departure delay falls from 27 to 23 intervals, showing that additional alternatives retain scheduling value after most objective and exposure gains have been captured. 

Operators can use the model to compare fleet schedules under different candidate route sets and assess the resulting changes in noise exposure, flight distance, passenger departure delay, and empty repositioning. These comparisons show whether adding routes provides meaningful operational benefits and how route choices, departure times, and aircraft movements must be coordinated to obtain those benefits while satisfying cruise-separation requirements. Its present scope comprises known requests, a homogeneous fleet, prescribed routes and altitudes, and discrete departure times. Separation is enforced during cruise at the same altitude, while charging, vertiport capacity, and takeoff and landing conflicts remain outside the model. Noise is represented by an additive exposure indicator. The operational findings concern one Stockholm demand realization and one nested route catalogue, so the number of alternatives needed to capture most benefits remains context-dependent.

Future work should improve computational scalability through more efficient cut-separation and branching strategies, together with the dynamic generation or selection of candidate flight paths. Incorporating battery management, charging, vertiport capacity, and separation during takeoff and landing would broaden the framework's operational applicability. A rolling-horizon extension could accommodate newly arriving requests and uncertain flight durations while updating service schedules and repositioning decisions. Further evaluation across cities, demand patterns, and noise-cost weights would help assess the robustness of the observed operational trade-offs, while aircraft-specific acoustic measurements could refine the exposure model.

\section*{Acknowledgements}
This research was funded by the Swedish Transport Administration (Trafikverket) under grant number TRV 2025/18378, within the project \textit{Bullermiljökonsekvensanalys av UAM i Stockholm (BUS)}. The authors thank the European Commission's Joint Research Centre and the Global Human Settlement Layer (GHSL) team for providing the GHS-POP R2023A gridded resident population dataset, specifically the 2020 population estimates at 3 arc-second resolution in WGS84. The authors also thank Stockholms stad, Milj\"of\"orvaltningen (City of Stockholm Environmental Administration), for providing the \textit{Bullerkartan 2022} maps of 24-hour equivalent background noise levels from road, rail, and aircraft traffic. The authors used Codex (OpenAI; GPT-5.5, GPT-5.6 Sol, and GPT-6 Astra) to assist with drafting and revising the manuscript and developing code to implement the proposed methods and generate synthetic test instances.

\appendix
\section{Construction of noise and flight-route costs}
\label{app:noise-cost}

We use the exponential impact indicator of \citet{BACLET2022103114}, Eqs.~(5) and~(8), and adapt the
route-aggregation approach of \citet{Venkataraman:2023}. Their approach weights linear noise increments by population and travel time; ours integrates exponential, population-normalized impact over flight distance.

\paragraph{Background noise and population}

The Stockholm \textit{Bullerkartan 2022} noise polygons are rasterized onto a $25\,\mathrm{m}$ grid. Each noise class is represented by its midpoint, with $37.5$ and $77.5\,\mathrm{dB(A)}$ assigned to classes below $40$ and above $75\,\mathrm{dB(A)}$, respectively. Where noise polygons overlap, the higher background sound level is assigned to the cell. Cells without background-noise data are excluded from the area available for route generation.

At ground receiver $i$, let $B_i$ denote the background sound level and $\rho_i$ the population density in people per $\mathrm{km}^{2}$. Density is obtained from GHS-POP R2023A population counts for 2020 by dividing each sampled count by the estimated area of its native 3-arcsec raster cell. Receivers are spaced $100\,\mathrm{m}$ apart, each representing an area of $a_R=0.01\,\mathrm{km}^{2}$.

\paragraph{Sound propagation and noise increment}

For an unmanned aerial vehicle (UAV) at horizontal position $s$ and altitude $h$, and a ground receiver at $q_i$, the slant distance and aircraft-only sound level are
\begin{align}
r_i(s,h)
&=\sqrt{\lVert s-q_i\rVert_2^2+h^2},
\label{eq:noise-slant-distance}\\
L_i^{\mathrm{UAV}}(s,h)
&=70-20\log_{10}\!\left(\frac{r_i(s,h)}{100}\right)
-0.6\frac{r_i(s,h)}{1000}.
\label{eq:noise-propagation}
\end{align}
Distances are in metres and sound levels in $\mathrm{dB(A)}$. The reference-level parameter is $70\,\mathrm{dB(A)}$, the reference distance is $100\,\mathrm{m}$, and atmospheric attenuation is $0.6\,\mathrm{dB/km}$ over the full slant distance.

Energetic addition gives the total sound level and its increase above the background:
\begin{align}
L_i^{\mathrm{tot}}(s,h)
&=10\log_{10}\!\left(
10^{B_i/10}+10^{L_i^{\mathrm{UAV}}(s,h)/10}
\right),
\label{eq:noise-total-level}\\
\Delta L_i(s,h)
&=L_i^{\mathrm{tot}}(s,h)-B_i
\notag\\
&=10\log_{10}\!\left(
1+10^{[L_i^{\mathrm{UAV}}(s,h)-B_i]/10}
\right).
\label{eq:noise-exceedance}
\end{align}
The increment is converted into an exponential indicator, with negligible values discarded:
\begin{equation}
g(\Delta)=
\begin{cases}
2^{\Delta/3}-1, & \Delta>0.05\,\mathrm{dB},\\
0, & \text{otherwise},
\end{cases}
\label{eq:noise-exponential-indicator}
\end{equation}
where $\Delta$ is expressed numerically in decibels.

\paragraph{Population-weighted impact field}

Let $\mathcal R(s)$ contain receivers within the study area with assigned background data, positive population density, and horizontal distance $\lVert s-q_i\rVert_2\leq900\,\mathrm{m}$. The impact at aircraft position $s$ is
\begin{equation}
I_h(s)
=a_R\sum_{i\in\mathcal R(s)}
\frac{\rho_i}{\rho_{\mathrm{ref}}}\,
g\!\left(\Delta L_i(s,h)\right),
\label{eq:noise-population-impact}
\end{equation}
where $\rho_{\mathrm{ref}}=1000$ people per $\mathrm{km}^{2}$. The density ratio is not capped, and $a_R$ accounts for the ground area represented by each receiver.

For each cruise altitude $h\in\{300,500\}\,\mathrm{m}$, the field is precomputed at the centres of the
$25\,\mathrm{m}$ grid cells. Each centre represents a possible aircraft position, whose impact is summed over the surrounding $100\,\mathrm{m}$ receiver grid. A route uses the field corresponding to its cruise altitude.

\paragraph{Integration along a route}

For segment $e$, let $\mathcal C(e)$ denote the cells it intersects, $\ell_{e,k}$ its length within cell $k$, and $I_{k,h}$ the impact at that cell's centre. Defining $\widehat I_h(s)=I_{k,h}$ within cell $k$ gives
\begin{align}
E_e(h)
&=\sum_{k\in\mathcal C(e)}\ell_{e,k}I_{k,h}
\notag\\
&=\int_e\widehat I_h(s)\,\mathrm{d}\ell
\approx\int_e I_h(s)\,\mathrm{d}\ell.
\label{eq:noise-segment-exposure}
\end{align}
Lengths are measured in metres. Splitting segments at grid boundaries integrates the piecewise-constant field $\widehat I_h$ exactly.

Consider a route--altitude option $\rho$ with planar path $P_\rho$ and cruise altitude $h_\rho$. Let $\mathcal E(P_\rho)$ denote the segments of this path. Its planar length and noise exposure are
\begin{equation}
\ell_\rho=\sum_{e\in\mathcal E(P_\rho)}\ell_e,
\qquad
E_\rho=\sum_{e\in\mathcal E(P_\rho)}E_e(h_\rho).
\label{eq:noise-route-components}
\end{equation}

\paragraph{Flight-route costs}

The preceding calculation gives noise exposure. To include it in the optimization objective, we combine
exposure and distance using weights $\beta$ and $\alpha$:
\begin{equation}
c_e(h)=\alpha\,\ell_e+\beta\,E_e(h),
\label{eq:noise-segment-cost}
\end{equation}
where $\ell_e$ is segment length. The implementation uses $\alpha=0.014$ and $\beta=0.0011$. These are fixed modeling assumptions, not values derived from the exposure integral or empirically calibrated prices.

Adding the cost of vertical operations gives
\begin{equation}
C_\rho
=\alpha\,\ell_\rho+\beta\,E_\rho
+8\frac{h_\rho}{300}.
\label{eq:noise-route-cost}
\end{equation}
The last term charges one climb and one descent at a prescribed cost of $4$ per time unit and a reference vertical speed of $300\,\mathrm{m}$ per time unit. This stored route cost is retained when flight durations are rounded to planning intervals.

The propagation parameters and spatial discretization are numerical modeling assumptions. The noise calculation covers fixed-altitude cruise under a static background, excluding takeoff and landing noise. Flight exposures are added individually, without constructing a joint sound field for simultaneous aircraft. Distance-based integration makes exposure independent of departure time and duration rounding.

\begin{landscape}
\section{Complete paired results}
\label{app:t90-full-results}
\setcounter{table}{0}
\captionof{table}{Complete paired results under a nominal 3,600 s search limit ($T=90$).}
\label{tab:t90-full-results}
\begingroup\centering
\setlength{\tabcolsep}{2.3pt}
\renewcommand{\arraystretch}{0.94}
\fontsize{7.8}{9}\selectfont
\newsavebox{\tNinetyCompactResults}
\begin{lrbox}{\tNinetyCompactResults}
\begin{tabular}{@{}lrrrrrrrrrrrrrr@{}}
\toprule
& \multicolumn{7}{c}{B123} & \multicolumn{7}{c}{Lazy} \\
\cmidrule(lr){2-8}\cmidrule(l){9-15}
Instance & LB & UB & $\Delta$ & Gap (\%) & Time (s) & Cuts & Nodes & LB & UB & $\Delta$ & Gap (\%) & Time (s) & Cuts & Nodes \\
\midrule
R60\_D8\_V16\_K4\_1 & 13,281.66 & 13,281.66\textsuperscript{*} & 0.00 & 0.000 & 1555 & \textbf{3,510} & \textbf{2,893} & 13,281.66 & 13,281.66\textsuperscript{*} & 0.00 & 0.000 & \textbf{251} & 4,329 & 57,891 \\
R60\_D8\_V16\_K4\_2 & 11,979.09 & 11,979.09\textsuperscript{*} & 0.00 & 0.000 & 200 & \textbf{1,194} & \textbf{819} & 11,979.09 & 11,979.09\textsuperscript{*} & 0.00 & 0.000 & \textbf{48} & 1,509 & 9,949 \\
R60\_D8\_V16\_K4\_3 & 12,108.05 & 12,108.05\textsuperscript{*} & 0.00 & 0.000 & 383 & \textbf{2,068} & \textbf{925} & 12,108.05 & 12,108.05\textsuperscript{*} & 0.00 & 0.000 & \textbf{92} & 3,568 & 16,531 \\
R60\_D8\_V16\_K4\_4 & 10,677.84 & 10,677.84\textsuperscript{*} & 0.00 & 0.000 & 104 & \textbf{1,320} & \textbf{324} & 10,677.84 & 10,677.84\textsuperscript{*} & 0.00 & 0.000 & \textbf{58} & 2,378 & 9,313 \\
R60\_D8\_V16\_K4\_5 & 10,943.81 & 10,943.81\textsuperscript{*} & 0.00 & 0.000 & 49 & \textbf{1,254} & \textbf{135} & 10,943.81 & 10,943.81\textsuperscript{*} & 0.00 & 0.000 & \textbf{36} & 1,268 & 9,197 \\
\hline
R60\_D8\_V24\_K8\_1 & \textbf{12,910.05} & \textbf{12,919.78} & \textbf{9.72} & \textbf{0.075} & 3601 & \textbf{4,392} & \textbf{5,323} & 12,704.34 & 12,961.19 & 256.85 & 1.982 & 3615 & 5,978 & 138,492 \\
R60\_D8\_V24\_K8\_2 & 11,433.06 & 11,433.06\textsuperscript{*} & 0.00 & 0.000 & 1209 & \textbf{2,165} & \textbf{2,563} & 11,433.06 & 11,433.06\textsuperscript{*} & 0.00 & 0.000 & \textbf{213} & 3,754 & 23,903 \\
R60\_D8\_V24\_K8\_3 & 11,350.64 & 11,350.64\textsuperscript{*} & 0.00 & 0.000 & 3244 & \textbf{5,193} & \textbf{3,324} & 11,350.64 & 11,350.64\textsuperscript{*} & 0.00 & 0.000 & \textbf{745} & 8,258 & 69,662 \\
R60\_D8\_V24\_K8\_4 & 9,919.59 & 9,919.59\textsuperscript{*} & 0.00 & 0.000 & 1546 & \textbf{3,094} & \textbf{3,590} & 9,919.59 & 9,919.59\textsuperscript{*} & 0.00 & 0.000 & \textbf{174} & 4,294 & 20,717 \\
R60\_D8\_V24\_K8\_5 & 10,628.16 & 10,628.16\textsuperscript{*} & 0.00 & 0.000 & 230 & \textbf{2,203} & \textbf{468} & 10,628.16 & 10,628.16\textsuperscript{*} & 0.00 & 0.000 & \textbf{73} & 2,479 & 8,781 \\
\hline
R60\_D10\_V20\_K8\_1 & 12,147.83 & 12,147.83\textsuperscript{*} & 0.00 & 0.000 & 1555 & \textbf{2,880} & \textbf{2,032} & 12,147.83 & 12,147.83\textsuperscript{*} & 0.00 & 0.000 & \textbf{341} & 3,475 & 28,877 \\
R60\_D10\_V20\_K8\_2 & 12,009.54 & 12,009.54\textsuperscript{*} & 0.00 & 0.000 & 620 & \textbf{1,913} & \textbf{1,037} & 12,009.54 & 12,009.54\textsuperscript{*} & 0.00 & 0.000 & \textbf{220} & 3,257 & 14,262 \\
R60\_D10\_V20\_K8\_3 & 13,996.42 & 13,996.42\textsuperscript{*} & 0.00 & 0.000 & 271 & \textbf{1,630} & \textbf{346} & 13,996.42 & 13,996.42\textsuperscript{*} & 0.00 & 0.000 & \textbf{153} & 2,705 & 9,495 \\
R60\_D10\_V20\_K8\_4 & 10,555.88 & 10,555.88\textsuperscript{*} & 0.00 & 0.000 & 532 & \textbf{1,589} & \textbf{999} & 10,555.88 & 10,555.88\textsuperscript{*} & 0.00 & 0.000 & \textbf{87} & 1,848 & 6,026 \\
R60\_D10\_V20\_K8\_5 & 12,758.13 & 12,758.13\textsuperscript{*} & 0.00 & 0.000 & 347 & 2,160 & \textbf{370} & 12,758.13 & 12,758.13\textsuperscript{*} & 0.00 & 0.000 & \textbf{93} & \textbf{1,766} & 5,222 \\
\hline
R60\_D10\_V30\_K4\_1 & 11,967.93 & 11,967.93\textsuperscript{*} & 0.00 & 0.000 & 71 & 1,270 & \textbf{241} & 11,967.93 & 11,967.93\textsuperscript{*} & 0.00 & 0.000 & \textbf{52} & \textbf{1,196} & 10,200 \\
R60\_D10\_V30\_K4\_2 & 11,938.84 & 11,938.84\textsuperscript{*} & 0.00 & 0.000 & 134 & \textbf{1,069} & \textbf{597} & 11,938.84 & 11,938.84\textsuperscript{*} & 0.00 & 0.000 & \textbf{57} & 1,555 & 10,290 \\
R60\_D10\_V30\_K4\_3 & 13,289.15 & 13,289.15\textsuperscript{*} & 0.00 & 0.000 & 66 & \textbf{1,189} & \textbf{214} & 13,289.15 & 13,289.15\textsuperscript{*} & 0.00 & 0.000 & \textbf{58} & 1,860 & 9,223 \\
R60\_D10\_V30\_K4\_4 & 10,274.08 & 10,274.08\textsuperscript{*} & 0.00 & 0.000 & 40 & 863 & \textbf{15} & 10,274.08 & 10,274.08\textsuperscript{*} & 0.00 & 0.000 & \textbf{17} & \textbf{853} & 2,248 \\
R60\_D10\_V30\_K4\_5 & 12,401.81 & 12,401.81\textsuperscript{*} & 0.00 & 0.000 & 22 & 816 & \textbf{11} & 12,401.81 & 12,401.81\textsuperscript{*} & 0.00 & 0.000 & \textbf{18} & \textbf{652} & 2,276 \\
\hline
R100\_D8\_V16\_K8\_1 & \textbf{23,155.52} & 23,954.99 & 799.46 & 3.337 & 3600 & \textbf{6,038} & \textbf{1,057} & 22,968.33 & -- & -- & -- & 3605 & 10,075 & 60,035 \\
R100\_D8\_V16\_K8\_2 & \textbf{22,543.72} & 23,462.69 & 918.97 & 3.917 & 3600 & \textbf{4,953} & \textbf{844} & 22,450.04 & -- & -- & -- & 3605 & 9,701 & 56,915 \\
R100\_D8\_V16\_K8\_3 & \textbf{22,531.94} & 23,070.54 & 538.60 & 2.335 & 3600 & \textbf{5,478} & \textbf{906} & 22,398.10 & -- & -- & -- & 3606 & 10,731 & 53,967 \\
R100\_D8\_V16\_K8\_4 & \textbf{19,739.14} & 20,256.09 & 516.95 & 2.552 & 3600 & \textbf{5,447} & \textbf{796} & 19,586.53 & -- & -- & -- & 3609 & 9,758 & 56,940 \\
R100\_D8\_V16\_K8\_5 & \textbf{20,224.80} & 20,888.32 & 663.53 & 3.177 & 3601 & \textbf{5,179} & \textbf{707} & 20,102.33 & -- & -- & -- & 3606 & 10,660 & 52,942 \\
\hline
R100\_D8\_V24\_K4\_1 & \textbf{20,805.46} & 21,588.78 & 783.33 & 3.628 & 3600 & \textbf{7,282} & \textbf{1,232} & 20,167.95 & -- & -- & -- & 3605 & 8,442 & 109,872 \\
R100\_D8\_V24\_K4\_2 & \textbf{19,739.06} & 20,161.90 & 422.85 & 2.097 & 3600 & \textbf{8,071} & \textbf{1,557} & 19,355.32 & -- & -- & -- & 3604 & 9,989 & 111,955 \\
R100\_D8\_V24\_K4\_3 & \textbf{19,568.66} & 20,439.73 & 871.07 & 4.262 & 3600 & \textbf{7,975} & \textbf{1,563} & 19,250.32 & -- & -- & -- & 3605 & 10,216 & 101,742 \\
R100\_D8\_V24\_K4\_4 & \textbf{17,453.04} & 17,467.99 & 14.95 & 0.086 & 3601 & 7,526 & \textbf{1,412} & 16,890.16 & -- & -- & -- & 3604 & \textbf{7,412} & 99,090 \\
R100\_D8\_V24\_K4\_5 & \textbf{17,606.16} & 18,394.23 & 788.06 & 4.284 & 3600 & \textbf{8,001} & \textbf{1,073} & 17,243.86 & -- & -- & -- & 3604 & 11,628 & 90,603 \\
\hline
R100\_D10\_V20\_K4\_1 & \textbf{20,977.69} & 20,977.69\textsuperscript{*} & 0.00 & 0.000 & 3113 & \textbf{5,761} & \textbf{1,942} & 20,752.26 & -- & -- & -- & 3607 & 6,497 & 99,233 \\
R100\_D10\_V20\_K4\_2 & \textbf{21,026.61} & 21,425.84 & 399.23 & 1.863 & 3600 & \textbf{5,744} & \textbf{1,442} & 20,878.34 & -- & -- & -- & 3605 & 7,423 & 78,612 \\
R100\_D10\_V20\_K4\_3 & \textbf{24,182.67} & 24,884.51 & 701.84 & 2.820 & 3600 & \textbf{5,937} & \textbf{1,000} & 23,971.45 & -- & -- & -- & 3604 & 7,933 & 71,810 \\
R100\_D10\_V20\_K4\_4 & 18,738.67 & 19,401.02 & 662.35 & 3.414 & 3600 & \textbf{5,264} & \textbf{1,426} & \textbf{18,752.99} & \textbf{18,752.99}\textsuperscript{*} & \textbf{0.00} & \textbf{0.000} & 2999 & 6,307 & 86,056 \\
R100\_D10\_V20\_K4\_5 & 22,105.95 & 22,108.09 & 2.14 & 0.010 & 3600 & \textbf{5,353} & \textbf{1,347} & \textbf{22,108.09} & 22,108.09\textsuperscript{*} & \textbf{0.00} & \textbf{0.000} & 1083 & 5,645 & 40,985 \\
\hline
R100\_D10\_V30\_K8\_1 & \textbf{19,153.07} & 19,392.02 & 238.95 & 1.232 & 3600 & \textbf{6,581} & \textbf{1,438} & 18,779.65 & -- & -- & -- & 3613 & 7,978 & 113,345 \\
R100\_D10\_V30\_K8\_2 & \textbf{19,625.18} & 19,947.97 & 322.79 & 1.618 & 3600 & \textbf{7,143} & \textbf{918} & 19,245.30 & -- & -- & -- & 3605 & 7,568 & 112,702 \\
R100\_D10\_V30\_K8\_3 & \textbf{22,114.49} & 22,850.49 & 735.99 & 3.221 & 3600 & \textbf{7,067} & \textbf{659} & 21,705.81 & -- & -- & -- & 3607 & 10,721 & 80,892 \\
R100\_D10\_V30\_K8\_4 & \textbf{17,372.46} & 17,535.73 & 163.28 & 0.931 & 3600 & \textbf{6,024} & \textbf{1,120} & 17,132.76 & -- & -- & -- & 3604 & 8,508 & 98,958 \\
R100\_D10\_V30\_K8\_5 & \textbf{20,116.23} & 20,473.67 & 357.44 & 1.746 & 3600 & \textbf{7,029} & \textbf{572} & 19,893.37 & -- & -- & -- & 3605 & 11,530 & 85,742 \\
\bottomrule
\end{tabular}
\end{lrbox}
\ifdim\wd\tNinetyCompactResults>\linewidth
  \sbox{\tNinetyCompactResults}{\resizebox{\linewidth}{!}{\usebox{\tNinetyCompactResults}}}
\fi
\ifdim\dimexpr\ht\tNinetyCompactResults+\dp\tNinetyCompactResults\relax>0.84\textheight
  \resizebox*{!}{0.84\textheight}{\usebox{\tNinetyCompactResults}}
\else
  \usebox{\tNinetyCompactResults}
\fi
\par\endgroup
\end{landscape}

\bibliographystyle{elsarticle-harv} 
\bibliography{cas-refs}

@ARTICLE{Cohen2021,
  author={Cohen, Adam P. and Shaheen, Susan A. and Farrar, Emily M.},
  journal={IEEE Transactions on Intelligent Transportation Systems}, 
  title={Urban Air Mobility: History, Ecosystem, Market Potential, and Challenges}, 
  year={2021},
  volume={22},
  number={9},
  pages={6074-6087},
  doi={10.1109/TITS.2021.3082767}}

@article{HILDEMANN2023102356,
title = {3D-flight route optimization for air-taxis in urban areas with Evolutionary Algorithms and GIS},
journal = {Journal of Air Transport Management},
volume = {107},
pages = {102356},
year = {2023},
issn = {0969-6997},
doi = {https://doi.org/10.1016/j.jairtraman.2022.102356},
url = {https://www.sciencedirect.com/science/article/pii/S0969699722001752},
author = {Moritz Hildemann and Judith A. Verstegen}
}

@article{SHON2025102720,
title = {An optimization framework for urban air mobility (UAM) planning and operations},
journal = {Journal of Air Transport Management},
volume = {124},
pages = {102720},
year = {2025},
issn = {0969-6997},
doi = {https://doi.org/10.1016/j.jairtraman.2024.102720},
url = {https://www.sciencedirect.com/science/article/pii/S0969699724001856},
author = {Heeseung Shon and Jinwoo Lee}
}

@article{Guo2026,
author = {Dongbo Guo and Enjian Yao and Rongseng Chen and Tianyu Zhang and Yongxiang Zhang},
title = {Optimising eVTOL vehicle scheduling problem in on-demand urban air mobility system with capacitated vertiports},
journal = {Transportmetrica A: Transport Science},
volume = {0},
number = {0},
pages = {2625143},
year = {2026},
publisher = {Taylor \& Francis},
doi = {10.1080/23249935.2026.2625143}}

@inproceedings{zhu2019pre,
  title={Pre-departure planning for urban air mobility flights with dynamic airspace reservation},
  author={Zhu, Guodong and Wei, Peng},
  booktitle={AIAA Aviation 2019 Forum},
  pages={3519},
  year={2019}
}

@article{GAO2024104740,
title = {Noise-aware and equitable urban air traffic management: An optimization approach},
journal = {Transportation Research Part C: Emerging Technologies},
volume = {165},
pages = {104740},
year = {2024},
issn = {0968-090X},
doi = {https://doi.org/10.1016/j.trc.2024.104740},
url = {https://www.sciencedirect.com/science/article/pii/S0968090X24002614},
author = {Zhenyu Gao and Yue Yu and Qinshuang Wei and Ufuk Topcu and John-Paul Clarke}
}

@article{KYI2026100231,
title = {Noise-aware urban drone delivery: Integrating demand modeling, routing optimization and spatial noise exposure assessment},
journal = {Journal of Urban Mobility},
volume = {9},
pages = {100231},
year = {2026},
issn = {2667-0917},
doi = {https://doi.org/10.1016/j.urbmob.2026.100231},
url = {https://www.sciencedirect.com/science/article/pii/S2667091726000488},
author = {Ei Phyu Kyi and Tao Feng}
}

@ARTICLE{Tan2024,
  author={Tan, Qichen and Zhong, Siyang and Qu, Renhao and Li, Yuhong and Zhou, Peng and Lo, Hong Kam and Zhang, Xin},
  journal={IEEE Intelligent Transportation Systems Magazine}, 
  title={Low-Noise Flight Path Planning of Drones Based on a Virtual Flight Noise Simulator: A Vehicle Routing Problem}, 
  year={2024},
  volume={16},
  number={6},
  pages={56-71},
  doi={10.1109/MITS.2024.3396430}}

@INPROCEEDINGS{Korringa2025,
  author={Korringa, Marielle and Snyder, Paul and Ullrich, Michael and Vidhyadharan, Sreejith and Lunnie, Robbie},
  booktitle={2025 Integrated Communications, Navigation and Surveillance Conference (ICNS)}, 
  title={Optimal Air Corridor Design for Efficient Integration of AAM Vehicles into the NAS}, 
  year={2025},
  volume={},
  number={},
  pages={1-7},
  doi={10.1109/ICNS65417.2025.10976768}}

@article{CHEN2026112176,
title = {Flight path planning with ground risk and noise mitigation for urban air mobility},
journal = {Reliability Engineering \& System Safety},
volume = {270},
pages = {112176},
year = {2026},
issn = {0951-8320},
doi = {https://doi.org/10.1016/j.ress.2025.112176},
url = {https://www.sciencedirect.com/science/article/pii/S0951832025013754},
author = {Yini Chen and Rao Fu and Shuai Jia}
}

@article{LEHOUILLIER2017696AirCollsionClique,
title = {Two decomposition algorithms for solving a minimum weight maximum clique model for the air conflict resolution problem},
journal = {European Journal of Operational Research},
volume = {256},
number = {3},
pages = {696-712},
year = {2017},
issn = {0377-2217},
doi = {https://doi.org/10.1016/j.ejor.2016.07.008},
url = {https://www.sciencedirect.com/science/article/pii/S0377221716305458},
author = {Thibault Lehouillier and Jérémy Omer and François Soumis and Guy Desaulniers}
}

@article{CAFIERI202367Continous,
title = {Mixed-integer nonlinear and continuous optimization formulations for aircraft conflict avoidance via heading and speed deviations},
journal = {European Journal of Operational Research},
volume = {310},
number = {2},
pages = {670-679},
year = {2023},
issn = {0377-2217},
doi = {https://doi.org/10.1016/j.ejor.2023.03.002},
url = {https://www.sciencedirect.com/science/article/pii/S0377221723001959},
author = {Sonia Cafieri and Andrew R. Conn and Marcel Mongeau}
}

@article{DIAS2022520,
title = {Disjunctive linear separation conditions and mixed-integer formulations for aircraft conflict resolution},
journal = {European Journal of Operational Research},
volume = {296},
number = {2},
pages = {520-538},
year = {2022},
issn = {0377-2217},
doi = {https://doi.org/10.1016/j.ejor.2021.03.059},
url = {https://www.sciencedirect.com/science/article/pii/S0377221721002897},
author = {Fernando H.C. Dias and Hassan Hijazi and David Rey}
}

@article{TS_ATC_review,
author = {Pelegr{\'i}n, Mercedes and D’Ambrosio, Claudia},
title = {Aircraft Deconfliction via Mathematical Programming: Review and Insights},
journal = {Transportation Science},
volume = {56},
number = {1},
pages = {118-140},
year = {2022},
doi = {10.1287/trsc.2021.1056}
}

@INPROCEEDINGS{Chour2023,
  author={Chour, Kenny and Pradeep, Priyank and Munishkin, Alexey A. and Kalyanam, Krishna M.},
  booktitle={2023 IEEE/AIAA 42nd Digital Avionics Systems Conference (DASC)}, 
  title={Aerial Vehicle Routing and Scheduling for UAS Traffic Management: A Hybrid Monte Carlo Tree Search Approach}, 
  year={2023},
  volume={},
  number={},
  pages={1-9},
  doi={10.1109/DASC58513.2023.10311314}}

@ARTICLE{Su2025,
  author={Su, Yu and Xu, Yan},
  journal={IEEE Transactions on Intelligent Vehicles}, 
  title={A Risk Assessment Method for Mid-Air Collisions in Urban Air Mobility Operations}, 
  year={2025},
  volume={10},
  number={2},
  pages={1327-1341},
  doi={10.1109/TIV.2024.3426915}}

@article{haba2025routing,
  title={Routing and scheduling optimization for urban air mobility fleet management using quantum annealing: R. Haba et al.},
  author={Haba, Renichiro and Mano, Takuya and Ueda, Ryosuke and Ebe, Genichiro and Takeda, Kohei and Terabe, Masayoshi and Ohzeki, Masayuki},
  journal={Scientific Reports},
  volume={15},
  number={1},
  pages={4326},
  year={2025},
  publisher={Nature Publishing Group UK London}
}

@article{WENG2026105477,
title = {Real-time traffic simulation and management for large-scale urban air mobility: Integrating route guidance and collision avoidance},
journal = {Transportation Research Part C: Emerging Technologies},
volume = {183},
pages = {105477},
year = {2026},
issn = {0968-090X},
doi = {https://doi.org/10.1016/j.trc.2025.105477},
url = {https://www.sciencedirect.com/science/article/pii/S0968090X25004814},
author = {Canqiang Weng and Can Chen and Jingjun Tan and Tianlu Pan and Renxin Zhong}
}

@article{TANG2021103324,
title = {Automated flight planning of high-density urban air mobility},
journal = {Transportation Research Part C: Emerging Technologies},
volume = {131},
pages = {103324},
year = {2021},
issn = {0968-090X},
doi = {https://doi.org/10.1016/j.trc.2021.103324},
url = {https://www.sciencedirect.com/science/article/pii/S0968090X21003302},
author = {Hualong Tang and Yu Zhang and Vahid Mohmoodian and Hadi Charkhgard}
}

@ARTICLE{Pooladsanj2026,
  author={Pooladsanj, Milad and Savla, Ketan and Ioannou, Petros A.},
  journal={IEEE Transactions on Intelligent Transportation Systems}, 
  title={Throughput Maximizing Takeoff Scheduling for eVTOL Vehicles in On-Demand Urban Air Mobility Systems}, 
  year={2026},
  volume={27},
  number={2},
  pages={1937-1951},
  doi={10.1109/TITS.2025.3640561}}

@article{ATAMTURK200040,
title = {Conflict graphs in solving integer programming problems},
journal = {European Journal of Operational Research},
volume = {121},
number = {1},
pages = {40-55},
year = {2000},
issn = {0377-2217},
doi = {https://doi.org/10.1016/S0377-2217(99)00015-6},
url = {https://www.sciencedirect.com/science/article/pii/S0377221799000156},
author = {Alper Atamtürk and George L. Nemhauser and Martin W.P. Savelsbergh}
}

@article{ZWANEVELD200114,
title = {Routing trains through a railway station based on a node packing model},
journal = {European Journal of Operational Research},
volume = {128},
number = {1},
pages = {14-33},
year = {2001},
issn = {0377-2217},
doi = {https://doi.org/10.1016/S0377-2217(00)00087-4},
url = {https://www.sciencedirect.com/science/article/pii/S0377221700000874},
author = {Peter J. Zwaneveld and Leo G. Kroon and Stan P.M. {van Hoesel}}
}

@article{ARAUJO2020104782,
title = {Strong bounds for resource constrained project scheduling: Preprocessing and cutting planes},
journal = {Computers \& Operations Research},
volume = {113},
pages = {104782},
year = {2020},
issn = {0305-0548},
doi = {https://doi.org/10.1016/j.cor.2019.104782},
url = {https://www.sciencedirect.com/science/article/pii/S0305054819302242},
author = {Janniele A.S. Araujo and Haroldo G. Santos and Bernard Gendron and Sanjay Dominik Jena and Samuel S. Brito and Danilo S. Souza}
}

@article{BRITO2021105176,
title = {Preprocessing and cutting planes with conflict graphs},
journal = {Computers \& Operations Research},
volume = {128},
pages = {105176},
year = {2021},
issn = {0305-0548},
doi = {https://doi.org/10.1016/j.cor.2020.105176},
url = {https://www.sciencedirect.com/science/article/pii/S0305054820302938},
author = {Samuel Souza Brito and Haroldo Gambini Santos}
}

@article{MOURA2025433,
title = {Compact formulations and valid inequalities for parallel machine scheduling with conflicts},
journal = {European Journal of Operational Research},
volume = {325},
number = {3},
pages = {433-443},
year = {2025},
issn = {0377-2217},
doi = {https://doi.org/10.1016/j.ejor.2025.04.006},
url = {https://www.sciencedirect.com/science/article/pii/S0377221725002565},
author = {Phablo F.S. Moura and Roel Leus and Hande Yaman}
}

@MISC{stockholmnoise,
  author={{Stockholms stad, Milj{\"o}f{\"o}rvaltningen}},
  title={{Bullerkartan 2022}: V{\"a}g, t{\aa}g- och flyg {Leq}},
  year={2022},
  howpublished={City of Stockholm Environmental Administration [Dataset]},
  url={https://www.arcgis.com/home/item.html?id=ce2b61853e594b7fa70de53eed7e9ef2},
  note={Road, rail, and aircraft 24-hour equivalent noise levels. Accessed 16 September 2026}
}

@article{ghsl2024,
author = {Martino Pesaresi and Marcello Schiavina and Panagiotis Politis and Sergio Freire and Katarzyna Krasnodębska and Johannes H. Uhl and Alessandra Carioli and Christina Corbane and Lewis Dijkstra and Pietro Florio and Hannah K. Friedrich and Jing Gao and Stefan Leyk and Linlin Lu and Luca Maffenini and Ines Mari-Rivero and Michele Melchiorri and Vasileios Syrris and Jamon Van Den Hoek and Thomas Kemper},
title = {Advances on the Global Human Settlement Layer by joint assessment of Earth Observation and population survey data},
journal = {International Journal of Digital Earth},
volume = {17},
number = {1},
pages = {2390454},
year = {2024},
publisher = {Taylor \& Francis},
doi = {10.1080/17538947.2024.2390454}
}

@MISC{ghspop,
  author={Schiavina, Marcello and Freire, Sergio and Carioli, Alessandra and MacManus, Kytt},
  title={{GHS-POP R2023A}: {GHS} population grid multitemporal (1975--2030)},
  year={2023},
  howpublished={European Commission, Joint Research Centre [Dataset]},
  doi={10.2905/2FF68A52-5B5B-4A22-8F40-C41DA8332CFE},
  url={https://doi.org/10.2905/2FF68A52-5B5B-4A22-8F40-C41DA8332CFE},
  note={2020 epoch, WGS84, 3 arc-second resolution. Accessed 16 September 2026}
}

@article{BACLET2022103114,
title = {From strategic noise maps to receiver-centric noise exposure sensitivity mapping},
journal = {Transportation Research Part D: Transport and Environment},
volume = {102},
pages = {103114},
year = {2022},
issn = {1361-9209},
doi = {https://doi.org/10.1016/j.trd.2021.103114},
url = {https://www.sciencedirect.com/science/article/pii/S1361920921004089},
author = {Sacha Baclet and Siddharth Venkataraman and Romain Rumpler and Robin Billsjö and Johannes Horvath and Per Erik Österlund}}

@article {Venkataraman:2023,
title = "Optimizing noise exposure in the Vehicle Routing Problem: A case study of last-mile freight deliveries in Stockholm",
journal = "INTER-NOISE and NOISE-CON Congress and Conference Proceedings",
parent_itemid = "infobike://ince/incecp",
publishercode ="ince",
year = "2023",
volume = "265",
number = "1",
pages = "6138-6149",
itemtype = "ARTICLE",
issn = "0736-2935",
url = "https://ince.publisher.ingentaconnect.com/content/ince/incecp/2023/00000265/00000001/art00016",
doi = "10.3397/IN_2022_0914",
author = "Venkataraman, Siddharth and Baclet, Sacha and Rumpler, Romain",
}
\end{document}